\documentclass[11pt]{article}

\usepackage[vmargin=3.0cm]{geometry}
\usepackage{import}
\usepackage[utf8]{inputenc}
\usepackage[T1]{fontenc}
\usepackage{amsmath}
\usepackage{amsthm}
\usepackage{amsfonts}
\usepackage{amssymb}
\usepackage[all]{xy}
\usepackage{enumitem}
\usepackage{color}
\usepackage{setspace}
\usepackage[hidelinks]{hyperref} 
\usepackage[noabbrev]{cleveref}
\usepackage{graphicx}

\usepackage{ifpdf}
    \ifpdf

\setlist{noitemsep,topsep=0pt,parsep=0pt,partopsep=0pt}

\newtheorem{thm}{Theorem}[section]
\newtheorem{lem}[thm]{Lemma}
\newtheorem{prop}[thm]{Proposition}
\newtheorem{cor}[thm]{Corollary}

\newtheorem{MainThm}{Theorem}
\newtheorem{MainCor}[MainThm]{Corollary}
\newtheorem*{thm*}{Theorem}

\theoremstyle{remark}
    \newtheorem{rmk}[thm]{Remark}

    \newtheorem*{rmk*}{Remark}
\theoremstyle{definition}
    
    \newtheorem{fact}[thm]{Fact}

\newcommand{\ZZ}{\mathbb{Z}}

\newcommand{\RR}{\mathbb{R}}
\newcommand{\CC}{\mathbb{C}}

\newcommand{\HH}{\mathbb{H}}
\newcommand{\NN}{\mathbb{N}}
\newcommand{\TT}{\mathbb{T}}

\newcommand{\A}{\mathcal{A}}

\newcommand{\C}{\mathcal{C}}

\newcommand{\CH}{\mathcal{C}H}

\renewcommand{\L}{\mathcal{L}}

\newcommand{\M}{\mathcal{M}}

\newcommand{\N}{\mathcal{N}}

\newcommand{\xx}{\mathbf x}

\newcommand{\ra}{\rightarrow}
\newcommand{\hra}{\hookrightarrow}

\newcommand{\lra}{\longrightarrow}

\newcommand{\mt}{\mapsto}

\renewcommand{\a}{\alpha}
\renewcommand{\b}{\beta}
\renewcommand{\d}{\delta} 
\newcommand{\e}{\varepsilon} 
\renewcommand{\i}{\iota}

\newcommand{\g}{\gamma}
\newcommand{\Ga}{\Gamma}
\renewcommand{\l}{\lambda}
\newcommand{\La}{\Lambda}

\newcommand{\Si}{\Sigma}
\renewcommand{\t}{\tau}
\renewcommand{\th}{\theta}
\newcommand{\w}{\omega}
\newcommand{\we}{\wedge}
\newcommand{\vphi}{\varphi}

\newcommand{\eu}{\text{eu}}

\newcommand{\im}{\text{Im }}
\newcommand{\Id}{\text{Id}}

\newcommand{\st}{\text{st}}

\newcommand{\sing}{\text{sing}}
\newcommand{\colim}{\text{colim }}

\newcommand{\dd}{\partial}
\newcommand{\ity}{\infty} 

\newcommand{\rest}{\restriction}

\renewcommand{\over}{\overset}
\newcommand{\under}[2]{\underset{#2}{#1}}
\newcommand{\sms}{\setminus}
\newcommand{\mr}{\mathring}

\newcommand{\<}{\left<}
\renewcommand{\>}{\right>}

\title{Sutured contact homology, conormal stops and hyperbolic knots}
\author{Côme Dattin}

\begin{document}

\maketitle

\begin{center}
\begin{abstract}
\noindent We apply the conormal construction to a hyperbolic knot $K \subset S^3$, and study the sutured contact manifold $(V, \xi)$ obtained by taking the complement of a standard neighbourhood of the unit conormal $\La_K \subset (ST^*S^3, \xi_\st)$. We show that the sutured Legendrian contact homology of a unit fiber $\La_0$, with its product structure, is a complete invariant of the knot (up to mirror). This can also be seen as the computation of the homology of the fiber in $ST^* S^3$, stopped at $\La_K$. Our main tool is, for any submanifold $N \subset M$, an explicit relationship between the complement of a unit conormal $\La_N$, and the unit bundle of $M \sms N$.
\end{abstract}
\end{center}

Given a {  contact manifold} $(V, \xi)$, a powerful invariant is its contact homology $CC(V, \xi)$, whose technical difficulties have been solved in \cite{Par15_CH_&_virtual}. Adding to the picture a closed {  Legendrian} submanifold $\La \subset (V, \xi)$, the corresponding invariant is its Chekanov-Eliashberg algebra, introduced by \cite{Che02_Dga_leg_links} for knots in $\RR^3$, and extended to the case of a closed Legendrian in a {  contactisation} $\RR \times P$ by \cite{EES02_LCH_R2n+1} \cite{EES05_LCH_PxR}. While one expect the construction to carry over for Legendrians in a compact contact manifold, the details have not been laid out yet.

Although those two settings (compact contact manifolds and contactisations) are the most commonly used in the literature, an in-between situation extending both points of view has been considered in \cite{CGHH10_Sutures} : the {  sutured} contact manifolds, which are closely related to the {  convex hypersurfaces} defined by Giroux \cite{Gir91_Convexity} \cite{Gir02_Dim3_vers+}.

In this paper we study a specific example, the complement of the unit conormal $\La_K$ of a knot $K \subset M$, in which the sutured setting arises naturally. Using the machinery of \cite{CGHH10_Sutures}, as well as the flexibility of the sutured setting, we compute the Legendrian homology of a unit fiber in $ST^*M \sms \La_K$ for an hyperbolic knot $K$. By adding the product structure, we obtain a complete invariant of hyperbolic knots, recovering the result of \cite{ENS16_Complete_knot_invrt} in a more restricted case.
As a byproduct of our construction, we also get an explicit relationship between the contact manifolds $ST^*M \sms \La_K$ and $ST^*(M\sms K)$, and between their invariants.

\subsection*{Sutured setting} 
Roughly speaking, a {\it sutured} contact manifold is a non-compact contact manifold $(V, \l)$ which, away from a compact, is the contactisation of a Liouville domain $(W, \b)$ :
\[ (V, \l) \simeq (\RR \times W, dt + \b) \quad \text{  away from a compact} \]
where we use the same notation for a Liouville domain $W$, with boundary, and its completion $W \cup (\RR^+ \times \dd W)$ which is non-compact.\\
Let us emphasise that, contrary to the usual compact case, we need to specify a contact form, and not only a contact structure. The results from \cite{CGHH10_Sutures} implies that the contact homology of such a manifold is a well-defined invariant. Moreover, they prove that given a sutured contact manifold, one can construct a (compact) contact manifold with convex boundary, and vice-versa. \\

Here are some examples of sutured contact manifolds that naturally arise using standard constructions : 
\begin{itemize}
    \item {\it  From a contactisation} : A contactisation $\RR \times W$ is obviously a sutured manifold. One can then apply any standard operation which only modify the interior : Legendrian or transverse surgery, Lutz twist... \\
    
    \item {\it Complement of a Legendrian} : As observed in \cite{CGHH10_Sutures}, given a (closed) Legendrian submanifold $\La$ in a (compact) contact manifold $(V, \xi)$, the complement of a standard neighbourhood of $\La$ has convex boundary. Applying the convex-sutured procedure, one obtain a sutured contact manifold $V_\La$. Moreover its sutured contact homology is a Legendrian invariant of $(\La, V, \xi)$. \\

    \item {\it Unit bundle of manifolds with boundary} : Starting with a closed smooth manifold $M$, its unit cotangent bundle\footnote{We call that way the spherization of the cotangent bundle. If we choose a metric $g$, there a contactomorphism between $ST^*M$ and $U_g M$, the unit bundle associated to that metric.} $ST^*M$ is a contact manifold. Moreover, if $K \subset M$ is a (smooth) submanifold, its unit conormal $\La_K  = ST^*_K M \subset ST^* M$ is a Legendrian submanifold. 

When $M$ is a manifold with boundary, it's unit bundle has convex boundary. More precisely, $ST^* M$ is the contactisation of $T^*(\dd M)$ away from a compact.
 Moreover, if $K$ is a submanifold with boundary $\dd K \subset \dd M$ (where the intersection $K \cap \dd M$ is transverse), then $\La_K$ is a sutured Legendrian, see \cite{Dat20_PhD_Homologies_leg_sut_applications} \cite{Dat22_Sutured_Legendrian_homology_2-braids} for more details. 

\end{itemize}

\subsection*{Main results and strategy} 
\paragraph{Complement of a unit conormal.}  
Consider a $k$-dimensional submanifold $K \subset M^n$, and $\La_0 \subset UM$ a Legendrian disjoint from $(UM)_{\rest K}$, the full unit bundle over $K$. The invariants of the complement of the conormal $\La_K$, and of the unit bundle of the complement of $K$, are related by :

\begin{MainThm} \label{MainThm.UM-UK}
There is a quasi-isomorphism 
\[LH \big(\La_0, UM \sms \La_K) \big) \simeq LH\big(\La_0, U (M \sms K )\big),\]
compatible with the product structure.
\end{MainThm}
Here we either assume that the contact form is hypertight (ie there is no contractible Reeb orbit), or that we can use an augmentation. In the general case, the Legendrian homology is a module over the contact homology of the manifold. Also note that $M \sms N$ (resp. $UM \sms \La_N$) should be understood as complement of a neighbourhood of $N$ (resp. $\La_N$), in particular they are manifolds with convex boundary. \\

The strategy is as follows. Topologically, the complement of the conormal $V_1 = UM \sms \La_K$ can be obtained from $V_0 = U(M \sms K)$ by gluing a {circular handle} $ST^*K \times S^{n-k-1} \times D^{2(n-k)}$ on the boundary. Moreover, the gluing locus is a neighbourhood of a submanifold of the dividing set of $\dd V_0$, where we think of $V_0$ and $V_1$ as manifolds with convex boundaries. To simplify the situation, we explicitly perturb the contact form to displace the suture, so the handle end up being glued away from the dividing set. We can interpret this operation as gluing two sutured contact manifolds, whose effects on contact homology and Legendrian homologies has been described in \cite{CGHH10_Sutures}. \\

When the submanifold is an hypersurface $\Si \subset M$, we obtain a slightly better result. The conormal $\La_\Si$ has two connected component, and we can choose one of them which we denote $\La_\Si^+$ (by coorienting the submanifold). If $\Si$  splits the manifold into $M = M_1 \cup_\Si M_2$, we have :

\begin{MainThm} \label{MainThm.LCH_Hypersurf_Stop}  If $\La_0$ is a Legendrian contained in $UM_1$, we get 
\[LC(\La_0, UM\sms \La_\Si^+) \simeq LC(\La_0, UM_1).\]
\end{MainThm}

\begin{rmk}[Singular submanifold]
\Cref{MainThm.UM-UK} also applies to a singular submanifold $N_\sing \subset M$, by defining it's conormal as the (cooriented) conormal of the boundary of a neighbourhood $\N(N_\sing)$. It is then a direct consequence of the previous corollary.
\end{rmk}

In the closed string setup, where we work with algebras generated by closed Reeb orbits, we also obtain the following result :
\begin{MainThm}\label{MainThm.CC_M-K}
 Consider a knot $K \subset M$, where $M$ is $n$-dimensional. Then 
  \[CC (UM \sms \La_K) \simeq CC\big(U(M\sms K)\big) \otimes \ZZ[(c_k)_{k \in \ZZ\sms 0}], \]
where the generators $c_k$ are all of degree $n+2$.
\end{MainThm}

\paragraph{Unit conormal of hyperbolic knots.}
Taking $M=S^3$, $K$ an hyperbolic knot and $\La =F^1_x$ a unit fiber in $US^3$, with $x \notin K$, we obtain by an explicit computation the following 
\begin{MainThm} \label{MainThm.complete_invariant}
The linearized Legendrian contact homology $LH(F^1_x, US^3 \sms U_K S^3)$, with its product structure, is a complete invariant of hyperbolic knots (up to mirror).
\end{MainThm}

By the previous result, we can compute the Legendrian homology in $U(S^3 \sms K)$, however we can't simply use the contact form induced hyberbolic metric, since it is not adapted to the sutured manifold (see \cref{sec.U(M-K)} for a more precise discussion). Instead we construct a family of contact forms which are adapted to the sutured manifold, and which "converge" toward the previous contact form. We then show that we can take the colimit, by an explicit use of Moser's trick. This implies that, in practice, we can actually use the contact form induced by the hyperbolic metric. Using the Morse flow trees from \cite{Ek05_Morse_flow_trees}, we compute the product structure, and show that the Legendrian homology is (roughly) the group ring of $\pi_1(S^3 \sms K)$.

Note that the previous colimit argument can also be applied to $UM \sms \La_K$, so it proves \cref{MainThm.complete_invariant} directly. However we believe that \cref{MainThm.UM-UK} holds its own interest. \\

\begin{rmk}
Using  our strategy for a general theorem seems technically more difficult. While it is known by \cite{Thu79_3-mfds} that a knot is either hyperbolic, a torus knot or a satellite knot (moreover, the complement of a torus knot is a Seifert fibered space), the computation of the homology of a fiber in $U(S^3 \sms K)$ is only simple in the hyperbolic case.  \\
On the other hand, one could use instead the cohomology of the space of based loops in $M\sms K$ \cite{AbSc04_FH_cotangent} \cite{Abo09_Fuk_loops}, or the string homology perspective \cite{CELN16_Knot_&_cord_alg}. However, one also need to recover the peripheral subgroup, by developing a more refined version of those theories for manifolds with boundary. 
 \end{rmk}

\paragraph{Conormal stops.}
Finally $UM$ is naturally filled by $DM$, and the previous theorem can be interpreted as a computation of the Floer homology of a fiber $F_x$ in $DM$, stopped at the conormal $U_K M$ (here $M$ doesn't have to be $S^3$, but we still assume that $K$ is hyperbolic and $x \notin K$).
\begin{MainCor} \label{MainCor.stopped_Floer}
The wrapped Floer homology of a fiber $F \subset DM$, stopped at $\La_K$, with its product structure, is isomorphic (as a unital ring) to $\ZZ[\bar \pi_1 (M \sms K)]$, 
where $\bar \pi_1$ denotes the group ring with the opposite law.
\end{MainCor}

\subsection*{Similar results}
\underline{\it Conormal of knots} :
Several results on knot conormals have appeared in the recent years, most notably Shende proved using sheaf techniques :
\begin{thm*}[\cite{She16_Conormal_torus_knot_inrt}]
 Given two knots $K_0$, $K_1$ in $\RR^3$, if $\La_{K_0}$ and $\La_{K_1}$ are Legendrian isotopic then the knots are smoothly isotopic, up to mirror and orientation reversal.
\end{thm*}

This result has been reproved using Legendrian contact homology :
\begin{thm*}[\cite{ENS16_Complete_knot_invrt}] Given a knot $K \subset \RR^3$ and $F^1_x$ the unit fiber over $x \notin K$, the Legendrian contact homology of $F^1_x \cup \La_K$, with its product structure, is a complete knot invariant.
 \end{thm*}
 The main tool in that paper is the isomorphism from \cite{CELN16_Knot_&_cord_alg} between the Legendrian homology of the conormal, and the string homology of the knot. The  remaining of the proof is of algebraic nature, to extract the knot group (with its peripheral subgroup) which is a known complete invariant by \cite{Wal68_Irred_3-mfds}. 

\begin{rmk}
For a knot $K \subset \RR^3$, the Legendrian homology of its conormal has been explicitly computed in \cite{EENS11_Knot_hom}, and coincide with the combinatorial invariant defined in \cite{Ng04_Framed_knot_hom}. 
\end{rmk}

Our \cref{MainThm.complete_invariant} can be seen as yet another proof of the previous results, in the case of an hyperbolic knot. However we want to emphasise that the computations in \cite{ENS16_Complete_knot_invrt} \cite{EENS11_Knot_hom} \cite{Ng04_Framed_knot_hom} rely on a presentation of the knot as the closure of a braid. By contrast, we use in a fundamental way the geometry of the complement of the knot, which is extremely simple in the hyperbolic case. \\

\underline{\it Stopped perspective} : In \cite{Syl16_Stops}, Sylvan defined the homology of a Lagrangian in a Liouville domain, stopped by a Legendrian in the boundary. This definition has then been extended to the context of Fukaya categories for any kind of stop in \cite{GPS17_Fuk_sectors}, using localisation of $\A_\ity$-categories. More recently a different perspective has been given in \cite{AsEk21_Chekanov-Eliashberg_dga_singular} for singular Legendrians, using the handles introduced in \cite{Avd12_Connect_sum_cobord}.

Given a Legendrian $\La \subset (V, \xi)$, one can interpret the sutured contact homology of $V \sms \La$ as the contact homology of $V$ stopped at $\La$, which should coincide with the previous definitions when the two objects are defined, see \cref{ssec.Stoped_POV} for a more precise discussion.

Another result about conormal stops appeared in the literature, in the usual symplectic setting : in \cite{As21_Fiber_Floer_conormal_stops} the following theorem, which generalise \cref{MainCor.stopped_Floer}, is proven.
\begin{thm*}[\cite{As21_Fiber_Floer_conormal_stops}] Consider a submanifold $K \subset M$, and a fiber $F_x$ disjoint from $\La_K$. Then there is an isomorphism between 
\begin{itemize}
    \item the Floer homology of $F \subset D^*M$, stopped at $\La_K$, with its $\A_\ity$-structure ;
    \item the homology of the space of loops in $M$ based at $x$, with the Pontryagin product.
\end{itemize}
\end{thm*}

\underline{\it Unit bundle of non-compact hyperbolic manifolds} : In \cite{BKO19_Formal_boundary_hyperb_knot_compl}, the cotangent bundle of the complement of an hyperbolic knot is studied from the point of view of Fukaya categories.
The most striking difference with the sutured approach is that, to define their invariants, they use the contact form induced by the hyperbolic metric on $M \sms K$. More precisely, they make use of the fact that, while the hyperbolic contact form is not adapted to the sutured contact manifold obtained from $U(M \sms K)$, its lift is adapted to the unit bundle of $\HH^3$, the standard hyperbolic space. The colimit argument of \cref{ssec.colim} shows that those two definitions are equivalent for Legendrian homology, see \cref{rmk.BKO} for more details.

\subsection*{Organisation}
In \cref{sec.Contact_symplectic} we review some classical results and constructions ib contact and symplectic geometry. In \cref{sec.LCH} we present the sutured framework : following \cite{CGHH10_Sutures}, we relate it to contact manifolds with convex boundary, and define the sutured contact homology of a Legendrian without boundary (for a more general construction see \cite{Dat20_PhD_Homologies_leg_sut_applications} \cite{Dat22_Sutured_Legendrian_homology_2-braids}). We also discuss the stopped perspective, and the correspondence between the different points of view. 

In section \cref{sec.U(M-K)} we start with a general presentation for unit bundles of manifolds with boundary, and then compute the Legendrian homology of a unit fiber in $U(M \sms K)$, for an hyperbolic knot $K$. In \cref{ssec.forms} we define a family of contact forms stretching the boundary, and we use a colimit argument in \cref{ssec.colim} to prove that we can actually use the contact form induced by the hyperbolic metric. Finally in \cref{ssec.prod} we compute the product on Legendrian homology.

In \cref{sec.UM-UK}, we relate $U(M\sms K)$ to $UM \sms \La_K$ via the gluing of a circular handle.
We explicit charts in which the computations are simpler, and we perturb the contact form so that the handles are glued away from the dividing set, which imply \cref{MainThm.complete_invariant}. 

Finally in \cref{sec.stopped_Floer}, we  present the higher dimensional situation, and prove our main theorems. \\


\paragraph{Acknowledgements :} This paper was part of a PhD thesis at the University of Nantes, and the author would like to thank his advisor Vincent Colin who introduced him to this subject, suggested several fascinating projects and was always ready for numerous questions. The author also thanks Alexandre Jannaud for the amount of interesting discussions along the years.  The author is currently supported by the Wallenberg foundation through the grant KAW 2019.0507, and is grateful for their help.

\tableofcontents 
~\\~\\

\section{Contact geometry and Legendrian homology} \label{sec.Contact_symplectic}
\subsection{Contact and symplectic geometry} 
Let us first recall some standard definitions. 
\paragraph{Contact geometry.} A {\it contact manifold} $(V, \xi)$ is a odd-dimensional manifold $V$ endowed with maximally non-integrable cooriented hyperplane field $\xi$, called a {\it contact structure}. A {\it Legendrian submanifold} $\La \subset (V^{2n+1, \xi})$ is a $n$-dimensional submanifold such that $T\La \subset \xi$.

Given a contact manifold $(V^{2n+1}, \xi)$, a {\it contact form} is a 1-form such that $\xi = \ker \l$. Note that a form defines a contact structure iff it satisfies the contact condition $\l \we (d\l)^n >0$. The {\it Reeb vector field} $R$ of a contact form is defined by the equations $\i_R d\l = 0$ and $\l(R)=1$. It is transverse to the contact structure, and its flow preserve it.

\paragraph{Symplectic geometry.} A {\it Liouville domain} $(W^{2n}, \b)$ is a even-dimensional manifold endowed with a 1-form $\b$ such that
\begin{itemize}
    \item $\w = d\b$ is a symplectic form, ie $(d\b)^n > 0$ ;
    \item the {\it Liouville vector field} $Y$, defined by $\i_Y d\b = \b $, is positively transverse to $\dd W$.
\end{itemize}
In particular the boundary $(\dd W, \b_{\rest \dd W})$ is a contact manifold. More precisely, on a neighbourhood $\N(\dd W)$ of the boundary $\dd W$ we have
\[ (\N(\dd W), \b) \simeq ((-\e,0]_\t \times \dd W, e^\t \b_{\rest \dd W}), \]
where the coordinate $\t$ is given by the flow of the Liouville vector field.

A {\it Lagrangian submanifold} $L \subset (W^{2n}, \b)$ is a $n$-dimensional submanifold such that $\w_{\rest TL} = 0$. It will be called 
\begin{itemize}
    \item {\it exact} if $\b_{\rest L}$ is an exact form, in other words if there exists a function $f : L \ra \RR$ such that $\b_{\rest L} = df$ ;
    \item {\it cylindrical} if its boundary is a Legendrian $\La \subset \dd W$, and in a neighbourhood we have 
    \[ \N(\dd L) \simeq (-\e,0] \times \La. \] It $L$ is also exact, we usually ask of the primitive $f$ to be constant at the boundary (note that $f$ is automatically locally constant).
\end{itemize}

\paragraph{From contact to symplectic manifolds, and vice versa.} Given a Liouville domain $(W, \b)$, its {\it contactisation} is the contact manifold $\big(\RR_t \times W, \ker(dt + \b)\big)$, and the Reeb vector field of the contact form $dt + \b$ is $\dd_t$. 
Given an exact Lagrangian $L \subset (W, \b)$, its Legendrian lift is \[\La = \{ (-f(x), x), x\in L  \} \subset \RR \times W,\] where $f$ is a primitive of $\b_{\rest L}$.

Given a contact manifold $(V, \l)$ (here we choose a specific contact form), its {\it symplectisation} is the symplectic manifold $\big(\RR_s \times V, d(e^s \l)\big)$, and if $\La \subset (V, \ker \l)$ is a Legendrian submanifold, then $\RR \times \La$ is Lagrangian.
More generally, given a Legendrian path $\La_t \subset (V, \xi_t), t \in [0,1]$, we can construct a Lagrangian cobordism in $\RR \times V$ from $(\La_0, V, \xi_0)$ to $(\La_1, V, \xi_1)$.

\subsection{Legendrian homology} \label{ssec.Leg_hom}
The contact manifolds usually studied in the literature are either compact or the contactisation of a Liouville domain, and the Legendrians are compact (although the non-compact case has been shown an increasing interest, eg in \cite{PaRu20_Aug_&_Immers_lagr_fill}). In those settings the {\it Legendrian homology} of a Legendrian $\La \subset (V, \xi)$ have been defined as follows (we make here some simplifying assumptions, see \cite{EES02_LCH_R2n+1} \cite{EES05_LCH_PxR} for more general constructions). Choose a contact form $\l$ such that 
\begin{itemize}
    \item there is no contractible periodic Reeb orbit ;
    \item the {\it Reeb chords} of $\La$, ie the Reeb trajectories joining two points of $\La$, whose set will be denoted $\C(\La)$, are non-degenerate : for any Reeb cord $c$ of action $T$, we have 
    \[(d\phi_R^T)(T_{c(0)} \La) \pitchfork T_{c(T)} \La.\]
    Those chords are graded after choosing some auxiliary data, see \cite{EES05_LCH_PxR} \cite{Bou09_Survey_contact}, and the degree of a chord $c$ will be noted $|c|$.
\end{itemize}
We also pick an adapted almost-complex structure $J$, ie a fiberwise endomorphism on the tangent space of the symplectisation $T(\RR_s \times V)$, such that :
\begin{itemize}
    \item $J$ is $s$-invariant ;
    \item on each $s$-level, $J$ preserve $\xi$ ;
    \item $J$ maps $\dd_s$ to $R_\l$.
\end{itemize}

Given Reeb chords $c_0, c_1, ..., c_k$, we denote $\M_{(V, \La)}(c_0, c_1...c_k; J)$ the set of $J$-holomorphic curves $U: (D_k, \dd D_k) \lra (\RR \times V, \RR \times \La)$, positively asymptotic to $c_0$ and negatively asymptotic to $c_1...c_k$, 
up to reparametrisation and $s$-translation. This space is of virtual dimension $|c_0| - \sum_1^k |c_k| - 1$, and has in general the structure of an orbifold. \\

Let $LC(\La, V, \l, J)$ be the graded unital algebra generated by Reeb chords, endowed with the differential defined on the generators by
\[ \dd c_0 = \sum_{k\geq 0}\under{\sum}{c_1,...,c_k \in \C(\La)} \# \M_{(V, \La)}(c_0, c_1...c_k;J) \]
and extended by the Leibniz rule $\dd (ab) = \dd a.b + (-1)^{|a|} a. \dd b$. The fact that $\dd^2 = 0$ result from the Gromov's compactness theorem, see \cite{BEHWZ03_Compactness_SFT}, and the homology of this complex will be denoted $LH(\La; V, \l, J) = \ker \dd / \im \dd$.
Furthermore, using once again Gromov compactness, this homology only depends on $(\La, V, \xi)$, see \cref{thm.LH_invrt} for a more precise statement.

\begin{rmk}
 One should actually add some homological coefficients in $H_1(\La)$, see \cite{EES05_LCH_PxR} for more details. However since our main example will be a $2$-dimensional sphere, it doesn't appear in our formulas.
\end{rmk}

\section{Sutured contact homology}  \label{sec.LCH}
The contact homology of a sutured contact manifold was defined in \cite{CGHH10_Sutures} using a completion into a non-compact contact manifold, in which the holomorphic curves still satisfy Gromov's compactness. This result still holds for a Legendrian without boundary embedded into the sutured manifold, hence its Legendrian sutured homology is also defined. We review here the construction from \cite{CGHH10_Sutures}.

\subsection{Sutured contact manifolds} \label{ssec.Sut_mfd}

Those manifolds with corners were first introduced by \cite{Gab83_Folia} in the topological setting, then used in the $3$-dimensional contact by \cite{CoHo05_Constr_control_Reeb} before being generalised in \cite{CGHH10_Sutures}. We note a slight change of nomenclature compared to that last paper : the convex sutured manifolds will just be called sutured manifolds, while the concave sutured manifold will be negative sutured manifold.

A {\it sutured contact manifold} is a collection $(V, \Ga, \l, \N_0, \psi)$ where ;
\begin{itemize}
    \item the pair $(V, \ker \l)$ is a compact, oriented, contact $(2n+1)$-manifold with corners ;
    \item $\Ga$ is $(2n-1)$-submanifold included in $\dd V$ ;
    \item $\N_0$ is a neighboorhood of $\Ga$, and $\psi$ is a diffeomorphism $\N_0 \lra (-\e, 0]_\t \times [-1,1]_t \times \Ga$ providing coordinates\footnote{As in \cite{CGHH10_Sutures} this product is actually oriented as $[-1,1]_t \times (-\e, 0]_\t \times \Ga$, but we want to think about $t$ (resp. $\t$) as the vertical (resp. horizontal) direction.} ;
\end{itemize}
satisfying the following conditions :
\begin{enumerate}
    \item the boundary of $V$ splits into
        $\dd V \simeq R_+ \under{\cup}{\{1\} \times \Ga} [-1, 1] \times \Ga \under{\cup}{\{1\} \times \Ga} R_-$,
    where the corners are exactly the gluing loci ;
    \item in coordinates, $\Ga = \{\t = t =0\}$ and 
        $ \dd V \cap \N_0 = \{t = 1\} \cup \{\t = 0\} \cup \{t=-1\}$,
    where once again the corners of $V$ are exactly the gluing loci ;
    \item \label{item.dfn_sut_R=Liouv} the pair $(R_+, \l_{\rest R_+})$ (resp. $(R_-, \l_{\rest R_-})$), oriented as the boundary of $V$ (resp. with reversed orientation), is a Liouville domain ;
    \item \label{item.dfn_sut_contact_f} on $\N_0$, in coordinates, we have $\l = C.dt + e^\t \l_\Ga$, where $\l_\Ga$ is a contact form on $\Ga$ (independent of $t$ and $\t$, and without term in $dt$ and $d\t$).
\end{enumerate}
The Liouville forms $\l_{\rest R_\pm}$ will be denoted $\b_\pm$, 
and the Reeb vector field of $(V, \l)$ (resp. $(\Ga, \l_\Ga)$ will be $R$ (resp. $R_\Ga$).
When the collection is unambiguous, a sutured manifold will only be noted $(V, \Ga, \l)$.

\paragraph{Some properties.}
We list some consequences from this definition which will be used afterwards, for more precise proofs we refer to \cite[§2]{CGHH10_Sutures}.

First of all, on the neighbourhood $\N_0$, the Reeb vector field is $C^{-1} \dd_t$ and the contact structure splits into 
\begin{equation} 
    \ker \l = \< \dd_\t, C R_\Ga -  e^\t \dd_t \> \oplus  \ker \l_\Ga,
\end{equation}
where $R_\Ga$ is the Reeb vector field of $(\Ga, \l_\Ga)$.
Also note that the conditions \cref{item.dfn_sut_R=Liouv} and \cref{item.dfn_sut_contact_f} imply that $(R_\pm, \b_\pm)$ are Liouville domains : the Liouville vector field being $\dd_\t$ it is outgoing, hence $R_\pm$ have no negative boundary.

Moreover the condition \cref{item.dfn_sut_R=Liouv} imply that the Reeb vector field of $\l$ is positively (resp. negatively) transerve to $R_+$ (resp. $R_-$). Using its flow we obtain neighbourhoods of $R_\pm$ contactomorphic to $((1-\e, 1]_t \times R_+, Cdt + \b_+)$ and $([-1, -1 +\e]_t \times R_-, Cdt + \b_-)$, extending the coordinate $t$ to a neighbourhood of all of $\dd V$. Note that the coordinate $\t$ can be recovered by integrating the Liouville vector field on each $t$-level, since on $\N_0$ it is $\dd_\t$.

\subsection{Contact manifold with convex boundary} \label{ssec.Mfds_cvx_boundary}
An hypersurface $\Si \subset (V, \xi)$ is {\it convex} if there exist a contact vector field $X$ (ie a vector field whose flow preserve $\xi$) transverse to $\Si$. A direct consequence is that the set 
\[ \Ga_X = \{ x \in \Si, X_x \in \xi_x \}  \]
is a contact submanifold of codimension $2$, dividing $\Si$ into two regions $R_\pm$. 
This configuration is actually a characterisation of a convex hypersurface :

\begin{lem} \cite[Lemme 2.2]{CGHH10_Sutures} \label{lem.cvx_adapt}
Let $\Sigma^{2n} \subset (V^{2n+1}, \xi)$ be a closed hypersurface. $\Sigma$ is convexe if and only if there exists an orientation of $\Sigma$, a submanifold $\Gamma^{2n-1} \subset \Sigma$ and a contact form $\lambda$  of kernel $\xi$ such that
\begin{itemize}
\item $(\Gamma, \xi \cap T\Gamma)$ is a contact manifold, oriented such that the contact form is positive.
\item $\Gamma$ splits $\Sigma$ alternating parts $R_\pm$, such $R_+$ induces on $\Gamma$ the previous orientation (ie $\Gamma = \dd R_+ = - \dd R_-$).
\item The Reeb vector field associated to $\l$ is positively (resp. negatively) transverse to $R_\pm$.
\end{itemize}
\end{lem}
A contact form satisfying those conditions will be called {\it adapted} to the convex boundary. The last condition is equivalent to requiring that $(R_+, \lambda)$ and $(-R_-, \lambda)$ are Liouville domains, where $R_+$ is oriented as $\Sigma$ and $-R_-$ denotes $R_-$ endowed with the opposite orientation. \\

Moreover, according to \cite[Lemma 4.1]{CGHH10_Sutures}, if $(V, \xi)$ is a manifold with convex boundary, and $X$ is a contact vector field transverse to the boundary, then there exist a contact form $\l$ and neighbourhoods of the suture $\N_0 \subset \N_1$ such that $(V \sms \N_0, \l, \N_1 \sms \N_0)$ is a sutured contact manifold.

\subsection{Sutured contact homology} \label{ssec.LH}

To define contact homology in a sutured manifold $(V, \Ga, \l)$, the first step is to complete the manifold with corners into a non-compact contact manifold $(V^*, \l^*)$. Using the above neighbourhood of the boundary, we glue the positive (resp. negative) contactisation of $(R_+, \b_+)$ (resp. $(R_-, \b_-)$) along $R_+$ (resp. $R_-$). The resulting boundary is $\RR \times \Ga$, on which we glue the contactisation of the positive symplectisation of $\Ga$ :
\[ V^* = V \cup [1, \ity)_t \times R_+ \cup (-\ity, -1]_t \times R_- \cup [0, \ity)_\t \times \RR_t \times \Ga,  \]
and the contact form is extended by $\l^* = dt + \b_\pm$ if $\pm t \geq 1$ and $\t < 0$, and by $dt + e^\t \l_\Ga$ if $\t \geq 0$. Note that for $\pm t>1$, each $t$-level is the completion of the Liouville domain $(R_\pm, \b_\pm)$. The Reeb vector field of $\l^*$ will still be denoted $R$.

\begin{figure}[h!]
\center
\def\svgwidth{8cm}
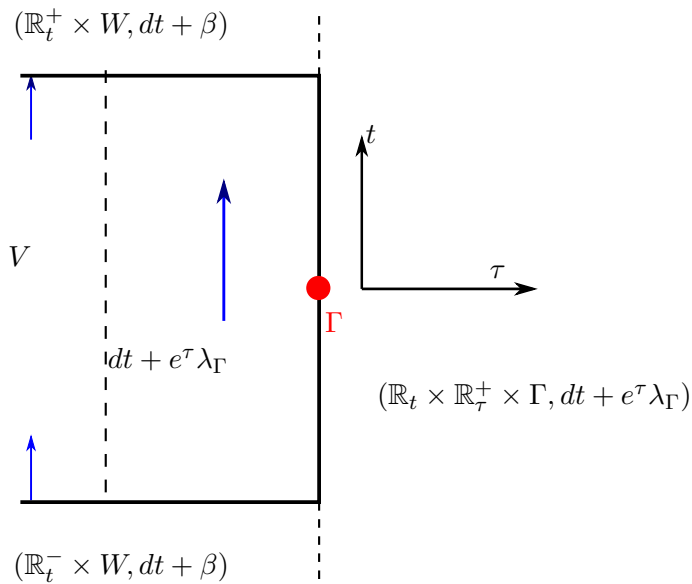
\caption{Completion of a sutured contact manifold. }
\label{fig.completion_sutured}
\end{figure}

\paragraph{Almost complex structure.} To achieve Gromov's compactness in this non-compact setting, the compatibility conditions must be strengthened. An {\it almost complex structure} on a manifold $W$ is a bundle endomorphism $J \in \text{End}(TW)$ such that $J^2 = - \Id$. On the symplectisation of the completion $(V^*, \l^*)$, ie $W = \RR_s \times V^*$ endowed with the Liouville form $\b = e^s \l^*$, an almost complex structure $J$ will be called {\it tailored} is it satisfy the following :
\begin{itemize}
    \item $J$ is $d\b$-admissible, ie $d\b(., J. )$ is a metric ; 
    \item $J$ is $\l^*$-adapted : it is $s$-invariant, preserve $\ker \l^*$ on each $s$-level, and map $\dd_s$ to $R$ ;
    \item $J$ is $t$-invariant on a neighbourhood of $V^* \sms \mr V$ ;
    \item the projection of $J$ to $T\hat R_\pm$ parallelly to $s$ and $t$ is adapted to the completion of $(R_\pm, \b_\pm)$ : this projection is $d\b_\pm$-admissible, and $\l_\Ga$-adapted on $\t\geq 0$.
\end{itemize}
Note that on the domain $\t \geq 0$, a tailored almost complex structure preserves $\xi_\Ga = \ker \l_\Ga$, and is determined by its restriction to $\xi_\Ga$ (which is independent of $s, t$ and $\t$), while on the domain $\{\pm t \geq 1, \t \leq 0\}$ it is determined by its projection to $T\hat R_\pm$.

\paragraph{Legendrian homology.}
Let $\La \subset (V, \l, \N_0)$ be a Legendrian without boundary, embedded in a sutured contact manifold. 
 In this paper we make the following assumptions, although more general settings are possible, see \cite{CGHH10_Sutures} \cite{EES02_LCH_R2n+1} \cite{EES05_LCH_PxR} :
\begin{itemize}
    \item the sutured manifold has no contractible Reeb orbit ;
    \item the group $H_1(V)$ is free ;
    \item the first Chern class $c_1(\ker \l)$ vanishes ;
    \item the Legendrian admits no (homotopically) contractible Reeb chord ;
    \item all Reeb chords are non-degenerate, see \cref{ssec.Leg_hom}
    \item $H_1(\La)$ is trivial.
\end{itemize}

Choose an tailored almost complex structure $J_0$, and denote by $LC(\La; V, \l, \N_0)$ the $\ZZ_2$-module generated by Reeb chords, ie Reeb trajectories joining two points of $\La$. Each chord is canonically $\ZZ_2$-graded, and after fixing a Maslov potential this grading lift to $\ZZ$, see \cite{EES02_LCH_R2n+1} for details. 
We can now define a differential on $LC(\La; V, \l, \N_0, J_0)$ by counting rigid holomorphic strips in the symplectisation of $(V^*, \l^*)$, as well as continuation morphisms.

\begin{thm}[\cite{CGHH10_Sutures}] \label{thm.LH_invrt}
 Counting holomorphic strips define a differential. Moreover a path $(\La_t, V_t, \l_t, \N_t, J_t)$  of Legendrians in sutured contact manifolds induces a quasi-isomorphism \[LC(\La_0; V_0, \l_0, \N_0, J_0) \lra LC(\La_1; V_1, \l_1, \N_1, J_1).\]
\end{thm}

In \cite{CGHH10_Sutures} only curves without boundary are examined, however since $\La$ has no boundary the bounds from \cite[Lemma 5.5 and Proposition 5.18]{CGHH10_Sutures} still holds, hence Gromov's compactness holds too (we can restrict to strips-counting when there is no contractible orbit or chord). Also note that, to define the continuation maps, a specific cobordism (between non-compact manifolds) is constructed, and the almost complex structure once again satisfy additional compatibility condition. In particular this map is {\it not} defined for any cobordism.

\paragraph{For manifolds with convex boundary.} 
For a contact manifold $(V, \xi)$ with convex boundary, we argue that its contact homology is well-defined : 
\begin{itemize}
    \item the space of possible dividing sets is connected : given two contact vector fields transverse to $\dd V$, interpolating between them yields contact vector fields which are still transverse to the boundary.
    \item once the dividing is fixed, the space of choices used to make the manifold sutured is contractible.
\end{itemize}
Hence the contact homology of the resulting sutured manifold $(\tilde V, \l)$ only depends on the starting manifold $(V, \xi)$, up to quasi-isomorphism. 
Consequently the contact homology of $(\tilde V, \l)$ 
will be denoted $\CH(V, \xi) = CH(\tilde V, \l)$. 

Similarly if $\La \subset (V, \xi)$ is a compact closed Legendrian which doesn't intersect the boundary $\dd V$, it's Legendrian homology is well-defined, and denoted $LC(\La; V, \xi) = LC(\La; \tilde V, \l$

We will make use of that observation in the later sections, as convex boundaries are more flexible and simpler to modify than the sutured ones.

\begin{rmk}
If $\La$ has boundary in $\dd V$, one can still define it's homology by choosing a well-behaved completion, see \cite{Dat22_Sutured_Legendrian_homology_2-braids}. However the invariance question is now slightly more subtle.
\end{rmk}

\paragraph{Product structure.} 
Consider an hypertight Legendrian $\La \subset (V, \xi)$, and assume there is no "bananas", ie there exist no holomorphic curves in $\RR\times V$, with boundary on $\RR \times \La$, and positively asymptotic to two Reeb chords (and no other asymptotes). 

Then its (linearized) homology $LH(\La, V)$ comes endowed with a product structure defined at the chain level by counting holomorphic curves in the symplectisation of $V$, with boundary in $\RR \times \La$, which are positively (resp. negatively) asymptotic to two (resp. one) Reeb chords :
\[ \mu(c, c') = \under{\sum}{c_- \in \C(\La)} \#\M(cc'; c_-) .c_-  \]

According to \cite{BEHWZ03_Compactness_SFT}, the moduli space $\M(cc'; c_-)$ has (virtal) dimension $|c| + |c'| - |c_-| - 1$, and admits a compactification by adding holomorphic buildings (where the curves in the symplectisation are quotiented by translation). More precisely, a $1$-dimensional moduli space admits a compactification of boundary
\begin{align*}
    \dd \M(cc'; c_-) = & \left( \under\bigcup{|c_0| = |c|-1}  \M(c, c_0) \times \M(c_0 c'; c_-) \right) \\
     & \cup \left( \under\bigcup{|c_0| = |c'|-1}  \M(c', c_0) \times \M(c c_0; c_-) \right) \\
     & \cup \left( \under\bigcup{|c_0| = |c_-|+1}  \M(cc'; c_0) \times \M(c_0; c_-) \right).
\end{align*}

Counting curves in the boundary of those moduli spaces (with $c$ and $c'$ fixed) yields the relation$0 =  \mu(\dd c, c') + \mu(c, \dd c') +\dd \mu(c,c')$ (although not necessary, we assume that the Legendrian complex has coefficients in $\ZZ_2$). For $c, c' \in \ker \dd$, we obtain
\begin{align*}
    \mu(c+ \dd a, c') & = \mu(c, c') + \mu(\dd a, c') = \mu(c, c') + \mu(a, \dd c') + \dd \mu(a, c') \\
     & = \mu(c, c') + \dd \mu(a, c')\qquad  \text{  because } \dd c' = 0,
\end{align*}
hence this product is well-defined (and associative) in homology. Moreover it is invariant along a Legendrian path : the moduli spaces of curves in a cobordism break in a similar way, and the chain map defined by the cobordism induces a morphism of (non-unital) rings map in homology, see \cite{Ekh06_Rational_SFT_Z2_cobord} for more details.

\begin{rmk}
Let us emphasize the necessity of being without bananas, the absence of those curves restricting the possible breakings. On the other hand, if the Legendrian is not hypertight, the product can still be defined by choosing an augmentation on the contractible chords (if it exists).
\end{rmk}

\subsection{Stopped point of view} \label{ssec.Stoped_POV}
In this part we assume that the sutured contact manifold is {\it balanced}, ie we have symplectomorphism $(R_pm, \l_{\rest R_\pm}) \simeq (W, \b)$. We give two other perspective inspired by the theory of stops, see \cite{Syl16_Stops} \cite{GPS17_Fuk_sectors} \cite{AsEk21_Chekanov-Eliashberg_dga_singular}.

\paragraph{Concave-convex procedure.}
Given a balanced sutured (or convex) contact manifold $V$, one can obtain a compact manifold $\tilde V$, with a "stop", by :
\begin{enumerate}
    \item gluing the contactisation $([-1,1]_z \times W, dz + \b)$ to the horizontal boundary of $V$ (note that $\{t=\pm 1\}$ is identified with $\{z = \mp 1\}$ using the above symplectomorphisms) ;
    \item filling the remaining holes by $D^2_{(r, \th)} \times \Ga$, endowed with a contact form  with expression $f(r) d\th + g(r) \l_\Ga$ for well-chosen functions. Most notably, we must have $f = 1$ and $g = e^{1-r}$ near $r=1$.
\end{enumerate}
This compact contact manifold comes with a Liouville hypersurface $\{z = 0\} \times W$, which will be seen as a {\it stop}. \\

The converse operation has been described in \cite{CGHH10_Sutures} : start with a contact manifold $(\tilde V, \tilde l)$ containing a Liouville hypersurface $W$, in other words the Reeb vector field is tranverse to $W$.
One can construct a sutured contact manifold $V_W$ as follows :
\begin{enumerate}
\item remove a standard neighbourhood $((-\e, \e)_z \times W, dz+\b) $ to obtain a contact manifold with corners. While it is not sutured\footnote{Such a manifold was called "concavely sutured" in \cite{CGHH10_Sutures}, although negatively sutured might be more appropriate, since a concave sutured boundary is also convex.} 
it can be made such ; 
\item modify the contact form such that two families of Reeb orbits appears, as depicted \cref{fig.ccv-cvx}, and remove a neighbourhood of $\{0\} \times \dd W$ (see \cite[§4.2]{CGHH10_Sutures} for more details) ;
\item finally smooth the extra corners, to obtain a sutured contact manifold with horizontal boundaries $R_\pm \simeq W$, and suture $\Ga = \dd W$.
\end{enumerate}

\begin{figure}[h!]
\center
\def\svgwidth{10cm} 
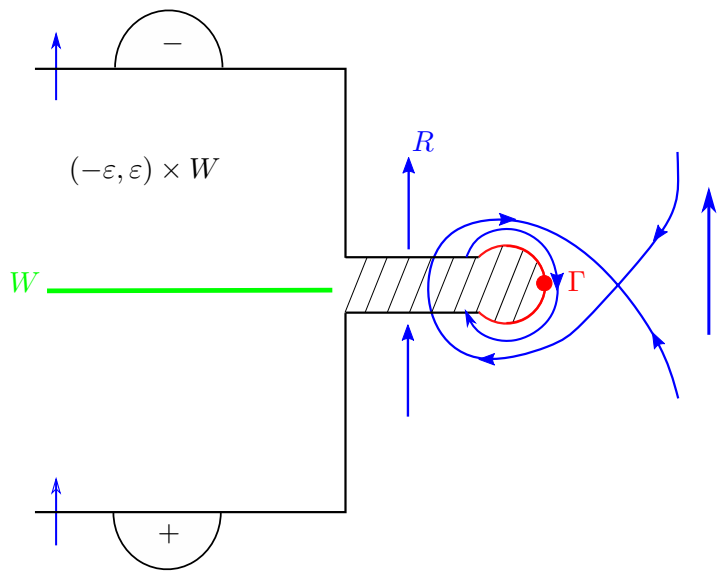
\caption{Modification of the contact form on the complement of $(-\e, \e) \times W$. In blue is the Reeb vector field, and in red the vertical boundary of the resulting sutured manifold.}
\label{fig.ccv-cvx}
\end{figure}

In \cite{CGHH10_Sutures}, the main example was for $W$ a standard thickening of a Legendrian. However this construction is can also be extended to a isotropic submanifold $S$, by picking for $W$ a Weinstein domain of skeleton $S$. In that situation, we can consider the pair $(\tilde V, \{0\} \times S)$, where $S$ is a isotropic stop in $\tilde V$. The sutured contact manifold obtained by the concave-convexe procedure will be denoted $V_S$.

\begin{rmk} In \cite{AsEk21_Chekanov-Eliashberg_dga_singular}, singular Legendrians are used as stops, however one must assume that they come with Weinstein thickenings. \end{rmk}

\paragraph{Stopping via handles.}
We now describe the point of view adopted in \cite{Avd12_Connect_sum_cobord} \cite{AsEk21_Chekanov-Eliashberg_dga_singular}. Consider a Liouville manifold $X$, with a Liouville hypersurface $(W, \b) \subset V = \dd_+ X$. Then by gluing a Liouville handle 
\[\RR^2_{(x,y)} \times W, 2x dy +y dx + \b  \]
on a neighbourhood of $W$, one obtain a Liouville cobordism $X_W$ with boundaries $\dd_- \tilde X = \dd_- X \sqcup \RR \times W$ and $\dd_+ \tilde X = Y_W$.

Morevoer, there are Reeb orbits in the positive boundary of the handle (more precisely, they are sitting in $\{x=y=0\} \times \dd W$), which correspond exactly to the "transversly hyperbolic" Reeb orbits appearing in the concave-convex procedure. In other words, the positive boundary of the handle correspond to the part greyed out in \cref{fig.ccv-cvx}.

\paragraph{Stopped contact homology.}
We can also define a notion of {\it stopped contact homology}, which should be related to the constructions of \cite{GPS17_Fuk_sectors}.

Given a Liouville hypersurface $W$ in a contact manifold $V$, its contact homology stopped at $W$ is
\[CC(X, W) = CC(X_W).\]
If $W$ is a Weinstein thickening of a (potentially singular) isotropic submanifold $S$, we will write $CC(X, S)$. 

Moreover if $\La_0 \subset V$ is a Legendrian disjoint from $W$, its stopped Legendrian homology is defined by 
\[LC(\La;V, W) = LC(\La; V_W).\]
Similarly if $W$ is a Weinstein thickening of a isotropic submanifold $S$, we will write $LC(\La; X, S)$.

\section{The unit bundle of the complement of the knot} \label{sec.U(M-K)}

We state some generalities on the unit tangent bundle of a manifold $M$ with boundary, before examining the case at hand. Choose coordinates on a neighbourhood of the boundary 
\[\N(\dd M)  \simeq (-\e, 0]_u \times \dd M,\]
as well as a metric $g_\dd$ on $\dd M$, and set $g = du^2 + g_\dd$, extended arbitrarily to $M$.
Then the unit tangent bundle 
$ V = S_g TM = \{(x, v) \in TM, g_x(v) = 1\}$
comes endowed with a contact form 
$(\l_g)_{(x, v)} = g_x(v, dv)$, whose Reeb vector field lift the geodesic flow of $(M,g)$. In other words Reeb trajectories project to geodesic paths, and the action is the length of the underlying path.

Moreover the vector field $\dd_u$ preserve the metric, hence it lifts to a contact vector field on $S_gTM$, which is transverse to the boundary. Consequently $V$ has convex boundary, and we have coordinates 
 \[ U_g \N(\dd M) = \{ (u, y; \nu, w) \in (-\e, 0]_u \times \dd M \times \RR \times T_y \dd M |\ \nu^2 + g_\dd(w, w) = 1     \}.\]
in which the boundary $\dd V$ splits into 
$R_\pm  = \{u = 0, \pm \nu >0 \} \simeq D^* (\dd M)$ (as Liouville domains), glued along the dividing set $\Ga = \{u = 0, \nu = 0\} \simeq S_{g_\dd} T(\dd M)$ (as contact manifolds).

This construction has a counterpart in the unit cotangent bundle $ST^* M$, whose contact structure is canonically defined. Indeed, any vector field on $M$ lifts to a contact vector field on $ST^* M$, and if it is transverse to $\dd M$ its lift is transverse to $\dd V$. This situation is summarised in the following property

\begin{prop}
If $M$ is a manifold with boundary, then its unit cotangent bundle $ST^* M$ has convex boundary, and its dividing set is the unit bundle of the boundary. Moreover, for any compact submanifold $N \subset M$ with boundary in $\dd M$, intersecting transversely the boundary, its conormal 
is Legendrian, of boundary $ST^*_{\dd N} (\dd M)$.
\end{prop}

The relevant submanifolds are:
\begin{align*}
    & V = ST^*M \qquad \dd V = D^* (\dd M) \under\cup\Ga D^* (\dd M) \qquad \Ga = ST^*(\dd M) \\
  & \qquad \qquad \qquad \La = ST^*_N M \qquad \dd \La = ST^*_{\dd N} \dd M.
\end{align*}

\begin{proof}
Indeed, if $X$ is a vector field on $M$, of flow $\phi^t$, the flow of the vector field lifted to $ST^*M$, denoted $\hat X = (X, 0)$, is 
\[ \hat \phi^t : (x, \a) \in ST^*M \mt (\phi(x), \phi_* \a).  \]
We claim that $\hat X$ preserve the contact structure: for $\a \in (T_x^* M \sms 0) / \RR_+^*$,
\begin{align*}
    d_{(x, \a)}\hat \phi^t(\xi_{(x, \a)}) &  = \{ (d_{(x, \a)}\hat \phi^t (v, w), v \in T_x M, w \in T_\a((T_x^* M \sms 0) / \RR_+^*), \a(v) =0 \} \\
    & = \{ (v', *) \in T_{(\phi^t(x), \phi^t_* \a)} (ST^*M), \a \circ (d_x \phi^t)^{-1}(v') = 0    \} \\
     & = \{ (v', *) \in T_{(\phi^t(x), \phi^t_* \a)} (ST^*M), \phi_* \a(v') = 0 \} \\
     &  = \xi_{(\phi^t(x), \phi_* \a)} = \xi_{\hat \phi^t(x, \a)}.
\end{align*}
If $X$ is transverse to $\dd M$, then $\hat X$ is transverse to the boundary of $ST^* M$, and the resulting dividing set is
\[ \Ga_{\hat X} = \{(x, \a), x \in \dd M, \a \in (T_x^* M \sms 0) / \RR_+^*, \a(X) = 0  \} \simeq ST^*(\dd M).  \]

Moreover, the boundary of the conormal $ST^*_N M \subset  (ST^* M, \xi_\st)$ is
\[\{ (x, \a) \in \dd M \times (T_x^* M \sms 0) / \RR_+^*, \a_{\rest TN} = 0   \}  \simeq ST^*_{\dd N} (\dd M). \] 
Indeed for $x \in \dd M$, a linear form on $T_x M$, vanishing on $T_x N$, is equivalent to a linear form on $T_x (\dd M)$, vanishing on $T_x(\dd N)$.
\end{proof}

Furthermore, if $X$ is a vector field on $M$ transverse to the boundary and tangent to $N$, the boundary $\dd \La$ is in the dividing set associated to $\hat X$. 

\begin{rmk}
To get a contact form adapted to the sutured contact manifold, we can take the one induced by the metric $f(u).(du^2 + g_\dd)$, where $g_\dd$ is a metric on $\dd M$ and $f$ is a strictly increasing function. Note that the metric $du^2+f(u).g_\dd$ also works (where $f$ is still increasing).
\end{rmk}
~\\

We now study the manifold $M_0 = S^3 \sms K$ obtained by removing a neighbourhood of a finite hyperbolic knot $K \subset M^3$, meaning that $M \sms K$ admits an hyperbolic metric of finite volume. In the rest of this section, we compute the Legendrian homology of a fiber $\La_0 \subset V_0 = ST^* M_0$.

\begin{thm} \label{thm.LH_U(M-K)=gp_ring} We have an isomorphism of (non-unital) rings
 $LH(\La_0, V_0) \simeq \ZZ_2[\bar \pi_1^*(M \sms K)]$, module generated by the $c_\g$ for $\g \in \pi_1(M \sms K) \sms \{1\}$, endowed with the product $(c_\g, c_{\g'}) \mt c_{\g'\g}$ (for $\g'\g\neq 1$).
\end{thm}

If the starting manifold is $M = S^3$, we get a complete knot invariant (up to miror) since we can recover the knot group by standard results, see also \cite{ENS16_Complete_knot_invrt} \cite{She16_Conormal_torus_knot_inrt}.
\begin{cor} \label{cor.complete_invrt}
For $M = S^3$, if $LH(\La_0, V_0(K)) \simeq LH(\La_0, V_0(K'))$ as (non-unital) rings, then $K \sim K'$ up to mirror.
\end{cor}

\begin{proof}
The assumed isomorphism extend to a (unital) ring isomorphism between the usual group rings $\ZZ_2[\pi_1(S^3\sms K)]$, obtained by adding an unit element $c_1$ with the obvious product extension. Moreover the group ring $\ZZ_2[\pi_1(S^3 \sms K)]$ is a complete invariant of the knot :
the groups $G_K = \pi_1(S^3 \sms K)$ being left orderable (\cite[Corollary 6.0.4]{Rol14_Order_group}), a ring isomorphism $\ZZ_2[G_K] \simeq \ZZ_2[G_{K'}]$ comes from a group isomorphism (\cite{LaRh68_Ring_group_left_order}, the idea being that the only invertible elements are monomials), hence the linearized homology, with its product structure, determines $G_K$ ;
furthermore, by Mostow rigidity theorem (recall that the knots are of finite volume), the group $G_K$ determines $S^3\sms K$, and by \cite{GoLu89_Knots_complements} two knots of $S^3$ are equivalent (up to miror) if and only if their complements are homeomorphic.
\end{proof}

\begin{rmk} To extend this result to any knot of $S^3$, one would need to recover the group  $\pi_1(M \sms K)$ as well as its peripherical subgroup, as in \cite{ENS16_Complete_knot_invrt} and \cite{She16_Conormal_torus_knot_inrt}.
\end{rmk}

\begin{rmk}
The following colimit argument also applies to the contact manifold $US^3 \sms \La_K$, hence it directly imply \cref{cor.complete_invrt}. However the statement \cref{MainThm.complete_invariant} is a lot more general. 
\end{rmk}

The rest of this section is dedicated to the proof of theorem \cref{thm.LH_U(M-K)=gp_ring}, using the following strategy : we first construct a sequence of contact forms on $V_0$ "stretching" toward the contact form induced by the hyperbolic metric on $M \sms K$. Using a colimit, we are able to lift the Legendrian to the unit bundle of $\HH^3$. There is no curve contributing to the differential by action reason, and we explicit the curves counted by the product using the Morse flow trees from Ekholm \cite{Ek05_Morse_flow_trees}.

\subsection{Contact forms} \label{ssec.forms}

Lets fix coordinates on the neighbourhood of $M_0$. By Margulis lemma (see eg \cite[Théorème D.3.3]{BenedPetro12_Hyperbolic}), there is a neighbourhood $\N(K)$ of the knot such that $\N(K) \sms K \simeq \RR^+_u \times \TT^2_z$, and on which the hyperbolic metric is $g_h = \frac{du^2 + g_{\TT^2}}{(C_0 + u)^2}$ where $C_0 > 0$ and  $g_{\TT^2}$ is the flat metric on the torus.

The manifold $V = U_{g_h}(\{u \leq 1\})$, endowed with the contact form $\l_h$ induced by the hyperbolic metric, has convex boundary. On a neighbourhood of the boundary we have coordinates  
\[ \{(u, z; \nu, \eta) \in [0,1] \times \TT^2 \times \RR \times \CC, \frac{\nu^2 + \eta^2}{C_0+u^2} = 1 \} \subset V \]
in which $\l_h = \frac{-1}{(C_0+u)^2} (\nu du + \eta \cdot dz)$. 
By the previous section the Legendrian homology of a fiber $\La_x \subset (V, \l_h)$ is well-defined. Note that a contact form adapted to the dividing set $\Ga = \{u=1, \nu=0\}$ has a dynamic similar to the form induced by an hyperbolic metric with a trumpet at the boundary, see \cref{fig.colim_metriq}.

We now construct  the family a contact forms $\l_T$ for $T> 0$, such that
\begin{itemize}
    \item for $u \leq  \frac46$, $\l_T$ comes from the hyperbolic  metric on $u \leq T$ : for some $\hat T \in [\frac T 6, \frac T 2]$, 
    \[(\{u \leq \frac46\}, \l_T) \simeq (\{u \leq \hat T\}, \l_h);\]
    \item for $\frac56 \leq u \leq 1$, $\l_T$ is independent of $T$.
\end{itemize}
Pick a smooth function $\vphi_T : [-\e,1] \ra [-\e, \ity)$ increasing, as well as $F_T : [-\e,1] \ra (0,\ity)$, such that 
\begin{itemize}
\item $\vphi_T(u) = u$ on $[-\e, 0]$
\item $\vphi_T' = T$ on $[\frac16, \frac26]$  
\item $\vphi_T' = 1$ on $[\frac36, 1]$,
\item $F_T = (C_0+\vphi_T(u))(C_0+u)$ for $u \leq \frac46$ 
\item $F_T = (C_0+u)^2$ for $\frac 56 < u <1$,
\end{itemize}
see figure \cref{fig.fct_etire} for an illustration. In particular $\hat T = \vphi(\frac46) \in [\frac T 6, \frac T 2]$. Set 
\[ \l_T = - \frac{1}{F_T} (\vphi_T' \nu du + \eta \cdot dz). \]

On $\{u \leq \frac 46\}$, $\l_T$ is the pullback of $\l_h$ under the lift of the application 
\[\phi_T : \{u \leq 1\} \ra \{u \leq \phi_T(\frac 4 6)\}\]
to the unit hyperbolic bundle, while on $\{\frac 56 \leq u \leq 1\}$ it is only proportional to that pullback.

\begin{figure}[h!]
\center
\def\svgwidth{16cm} 
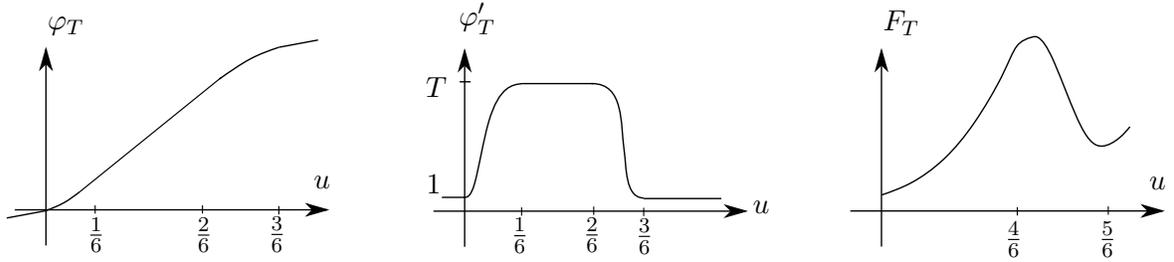
\caption{Functions used to stretch the contact form near the boundary.}
\label{fig.fct_etire}
\end{figure}

\begin{figure}[h!]
\center
\def\svgwidth{13cm} 
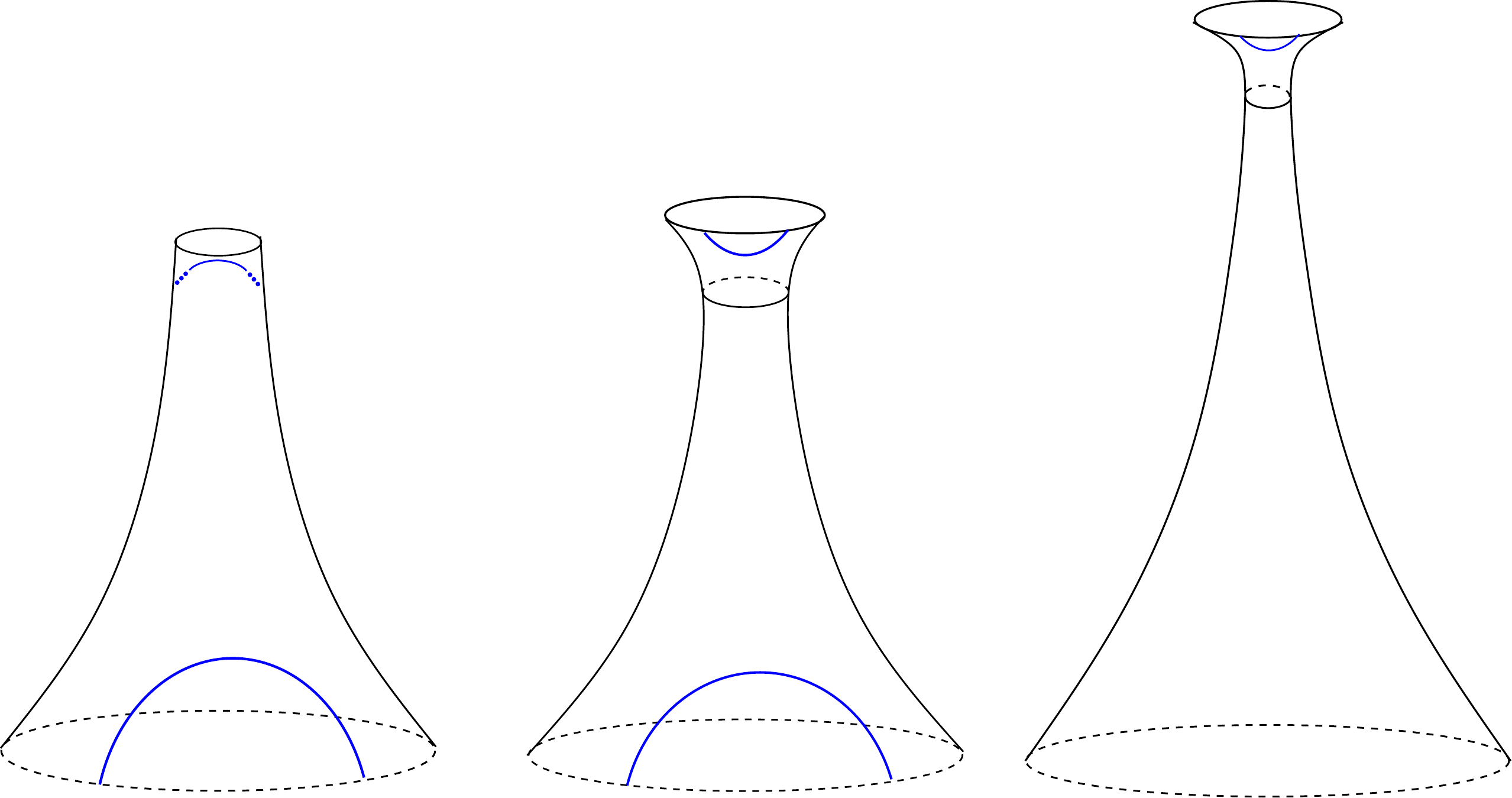
\caption{The metric corresponding to the contact forms $\l_h$, $\tilde{\l}_h$ and $\tilde{\l}_T$ (where the last two forms are adapted to the suture), with geodesics almost tangent to the boundary (in blue). }
\label{fig.colim_metriq}
\end{figure}

\subsection{Colimit} \label{ssec.colim}		

We now show that it is enough to consider chords lifting geodesics in $(M \sms K, g_h)$, which will gives us the desired homology : 
\begin{enumerate}
    \item the geodesic loops based in $x$ are in bijection with the non-trivial elements of $\pi_1(M\sms K, x)$, and the corresponding chords have all degree one (see the next section) ;
    \item a pseudo-holomorphic curve positively asymptotic to the chord (corresponding to) $\g \in \pi_1(M \sms K)$ must be negatively asymptotic to chords $\g_1...\g_k$ such that $\g_1 \cdot ... \cdot \g_k = \g$ in $\pi_1(M\sms K)$. This is impossible for action reasons : the $\l_h$-action of a chord is the $g_h$-length of the underlying geodesic, which is minimized by the unbroken geodesic corresponding to $\g$, see figure \cref{fig.geod_H2}.
\end{enumerate}

\begin{rmk} The contact forms $\l_T$ have no contractible Reeb orbit (even after making it a sutured contact manifold), hence we can discard them. Also note that there is no contratible chord either, hence the strip Legendrian homology is well defined, and the second step above isn't required.  
\end{rmk}

For $\g \in \pi_1(M \sms K)$, $c_\g$ will denote the geodesic loop in $(M\sms K, g_h)$ based in $x$, as well as the lift as a the Reeb chord. The geodesic lifted to $\HH^3$, as well as the lift of the Reeb chord to $U H^3$, will be denoted $\hat c_\g$. 

\begin{lem} \label{lem.length_geods} There exists $C > 0$ such that for any $A >0$ and  $T \geq T(A) = C e^A$, the $\l_T$-chords of action bounded by $A$ correspond to geodesic loops of $(M\sms K, g_h)$ of length bounded by $A$. 
\end{lem} 

\begin{proof}
Firstly, a Reeb trajectory of $(V_0, \l_T)$ staying in $\{u \leq \frac46\}$ correspond to a geodesic of $(M \sms K, g_h)$. Moreover the length of a $g_h$-geodesic joining $x$ to the hypersurface $\{u=u_0\}$ is greater than $\ln(1+\frac{u_0}{C_0})$, hence by taking $T(A)$ such that $A = d_{g_h}(x, \{u = \frac{T(A)}6\})$ we obtain the desired result.
\end{proof}

Hence as discussed above, for $T \geq T(A)$ the bounded Legendrian complex is
\[LC^A(\La_x, V_0, \l_T) = \ZZ[\bar\pi_1^{*,A}(M\sms K)],\]
with trivial differential. Here $\l_T$ is modified near the boundary into a contact form adapted to the dividing set, independently from $T$ and such that it doesn't create any contractible Reeb orbit.   

\begin{figure}[h!]
\center
\def\svgwidth{9cm} 
\begingroup%
  \makeatletter%
  \providecommand\color[2][]{%
    \errmessage{(Inkscape) Color is used for the text in Inkscape, but the package 'color.sty' is not loaded}%
    \renewcommand\color[2][]{}%
  }%
  \providecommand\transparent[1]{%
    \errmessage{(Inkscape) Transparency is used (non-zero) for the text in Inkscape, but the package 'transparent.sty' is not loaded}%
    \renewcommand\transparent[1]{}%
  }%
  \providecommand\rotatebox[2]{#2}%
  \newcommand*\fsize{\dimexpr\f@size pt\relax}%
  \newcommand*\lineheight[1]{\fontsize{\fsize}{#1\fsize}\selectfont}%
  \ifx\svgwidth\undefined%
    \setlength{\unitlength}{352.74450347bp}%
    \ifx\svgscale\undefined%
      \relax%
    \else%
      \setlength{\unitlength}{\unitlength * \real{\svgscale}}%
    \fi%
  \else%
    \setlength{\unitlength}{\svgwidth}%
  \fi%
  \global\let\svgwidth\undefined%
  \global\let\svgscale\undefined%
  \makeatother%
  \begin{picture}(1,0.51533516)%
    \lineheight{1}%
    \setlength\tabcolsep{0pt}%
    \put(0,0){\includegraphics[width=\unitlength,page=1]{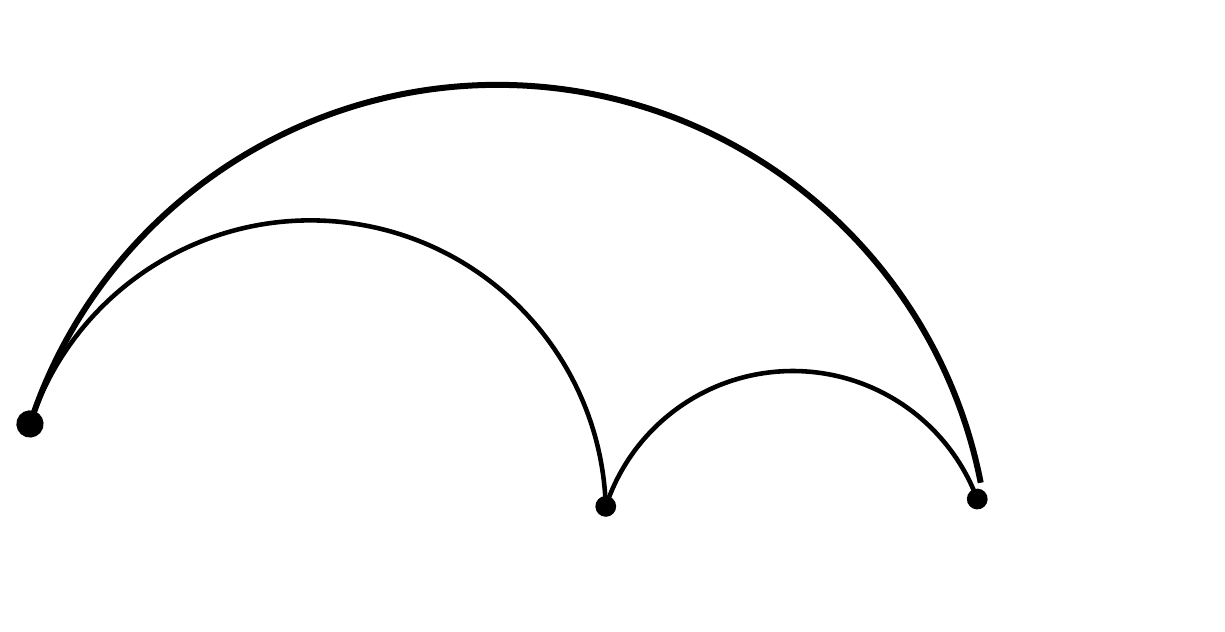}}%
    \put(-0.00423576,0.02870953){\color[rgb]{0,0,0}\makebox(0,0)[lt]{\lineheight{1.25}\smash{\begin{tabular}[t]{l}$z$\end{tabular}}}}%
    \put(0.41647735,0.01061437){\color[rgb]{0,0,0}\makebox(0,0)[lt]{\lineheight{1.25}\smash{\begin{tabular}[t]{l}$\g z$\end{tabular}}}}%
    \put(0.78290488,0.04228091){\color[rgb]{0,0,0}\makebox(0,0)[lt]{\lineheight{1.25}\smash{\begin{tabular}[t]{l}$\g'\g z$\end{tabular}}}}%
    \put(0.21743032,0.25942314){\color[rgb]{0,0,0}\makebox(0,0)[lt]{\lineheight{1.25}\smash{\begin{tabular}[t]{l}$c_\g$\end{tabular}}}}%
    \put(0.51600089,0.2232328){\color[rgb]{0,0,0}\makebox(0,0)[lt]{\lineheight{1.25}\smash{\begin{tabular}[t]{l}$c_{\g'}$\end{tabular}}}}%
    \put(0.38933461,0.47656555){\color[rgb]{0,0,0}\makebox(0,0)[lt]{\lineheight{1.25}\smash{\begin{tabular}[t]{l}$c_{\g \g'}$\end{tabular}}}}%
  \end{picture}%
\endgroup%

\caption{Geodesics in the hyperbolic half-space}
\label{fig.geod_H2}
\end{figure}
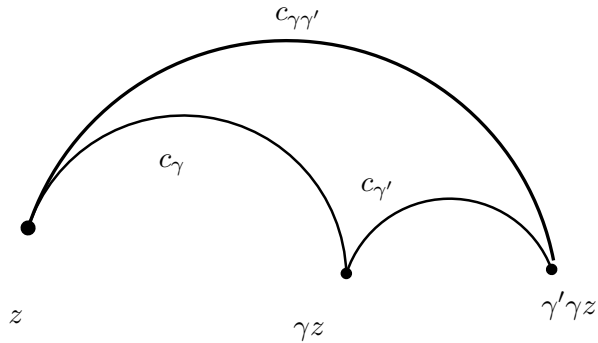

We now apply the following lemma to ensure that the colimit recovers the linearized homology of $(\La_x, V_0, \l_h)$. Note that the reference provided is stated for contact homology of compact contact manifolds, and still holds here since the contact forms $\l_T$ are fixed near the boundary (and the Legendrian is fixed).

\begin{lem}\cite[Lemma 3.1.7 \& Theorem 3.2.1]{CGH10_ECH_&_OBD} \label{lem.ColimFormes} 
Let  $(V, \l)$ be a contact maniflod containing a Legendrian $\La$, $T_n$ a sequence of reals, and $A_n : V \ra (0, \ity)$ a sequence of smooth and positive functions. Then 
\begin{enumerate}
\item If $\frac{A_{n+1}}{f_{n+1}} > \frac{A_n}{f_n}$, the maps \[LC^{\leq A_n}(\La, V, f_n \l) \lra LC^{\leq A_{n+1}}(\La, V, f_{n+1} \l)\] are well-defined using exact Lagrangian cobordisms, and compatibles.
\item if $Q_n = \frac{A_n}{f_n} \under{\lra}{n \ra \ity} \infty$, we have $\under{\colim}{n \ra \ity} LC^{\leq A_n}(\La, V, f_n \l) \simeq LC(\La, V, \l)$.
\end{enumerate}
\end{lem}

\paragraph{The colimit argument.}
In the remaining of this section, we prove that the last lemma ca be applied to the sequence $LC^A(\La_x, V_0, \l_{T(A)})$, which imply that $LH(\La_0, V_0)$ is the desired module. We use Gray's theorem in an explicit way to obtain the necessary estimations.
 \begin{lem} \label{lem.bound_mu_T}
 There is diffeomorphism $\psi_T : V_0 \ra V_0$ such that, for all $T \geq 1$,
 \[ (V_0, \l_T) \over{\psi_T}\simeq (V_0, \mu_T \l_1). \]
Moreover we have the following estimation
\begin{equation} \label{eqn.estim_mu_T}
    \frac CT \leq \mu_T \leq C, 
\end{equation}
where $C$ is independent from $T$, and
\begin{equation} \label{eqn.diverg_Q_A}
    Q_A = \inf_{V_0} A/\mu_{T(A)} \under{\lra}{A \ra \ity} +\ity.
\end{equation}
\end{lem}
\cref{eqn.estim_mu_T} implies that, for a sequence $A_n$ increasing fast enough, the first condition of \cref{lem.ColimFormes} is satisfied, since it yields
\[ \frac A C \leq \frac A {\mu_{T(A)}}  \leq CAe^A. \]
Hence the cobordism maps are well defined, and by \cref{eqn.diverg_Q_A} the second condition of \cref{lem.ColimFormes} is satisfied as well.

\begin{proof} 
We will estimate the diffeomorphisms $\psi_T$ and the functions $\mu_T$ by expliciting the Moser's trick. \\

{\it { Step 1} : Moser's trick} \\
We apply the Moser's trick to the path $F_T \l_T = - \vphi'_T \nu du - \eta dz, T \geq 1$, to obtain an isotopy $\psi_T$ such that  
\[\psi_T^*(\vphi'_T \nu du + \eta dz) = \bar \mu_T( \nu du + \eta dz)\]
where $\psi_T$ is the flow of a vector field $X_T$ and $\bar \mu_T$ is a positive function, given by
\begin{align*}\label{eqn.Hyperb.Moser}
    \dd_T (\ln \bar \mu_T) & = \dd_T  \bar \l_T (R_{\bar \l_T}) \circ \psi_T \\
     (\i_{X_T} d\bar \l_T)_{\rest \xi_T} & = \left( \dd_T \bar \l_T(R_{\bar \l_T}).\l_T - \dd_T \bar \l_T \right)_{\rest \xi_T}.
\end{align*}
Notice that the contact form is invariant for $\frac56 < u \leq 1$, hence $\psi_T = \Id$ and $\bar\mu_T = 1$ on that region. The function $\mu_T$ takes the expression $\mu_T = \frac{\bar \mu_T}{\psi_T^* F_T}$ (which is also $1$ for $\frac56 < u \leq 1$). \\ 

{\it {Step 2} : Computing $X_T$} \\
\begin{fact}
The Reeb vector field of $- \vphi'_T \nu du - \eta dz$ is 
$R_{\bar \l_T} = - B (\nu \dd_u + \vphi_T' \eta.\dd_z + 2  \dd_\nu)$,
where $B = \frac{1}{\vphi_T' (C+u)^2} \geq 0$.
\end{fact}
\begin{proof}
Indeed $d(F_T\l_T) = \vphi_T' du \we d\nu + dz \we d\eta$, and $V_0$ is defined by the equation $\frac{\nu^2 + \eta^2}{(C+u)^2}=1$ hence \begin{align*}
0 & = d((\nu^2 + \eta^2)/(C+u)^2) \\
& = (C+u)^{-2}(\nu d\nu + \eta d\eta) - 2 (C+u)^{-4} (\nu^2 + \eta^2) du \\
 & = (C+u)^{-2} \left( \nu d\nu + \eta d\eta - 2 du \right)
 \end{align*}
We can easily check that the vector field $\nu \dd_u + \vphi_T' \eta.\dd_z + 2 \dd_\nu$ is tangent to $V_0$,
and that it directs the Reeb vector field  : 
\begin{align*}
    \i_{\nu \dd_u + \vphi_T' \eta.\dd_z + 2 \dd_\nu} d(F_T \l_T) & = \nu \vphi'_T d\nu  - \vphi'_T \eta.d\eta - 2 \vphi'_T  du \\
    & = \vphi'_T (\nu d\nu + \eta d\eta - 2 du) = 0.
\end{align*}
Finally the normalisation condition determines $B$. 
\end{proof}

\begin{fact}
We have $du (X_T) = 0$, hence $\psi_T$ preserve the $u$ coordinate.  
\end{fact}

\begin{proof}
Indeed the contact hyperplan $\ker \l_T$ is the intersection of two kernels :
\[ \xi_T = \ker (2du - \nu d\nu - \eta\cdot d\eta) \cap \ker (\vphi_T' \nu du + \eta \cdot dz),  \]
and is generated by the vectors
\begin{align*}
    & V_i = \eta_i \dd_\nu - \nu \dd_{\eta_i} \\
    & W_i = \eta_i \dd_u - \vphi_T' \nu \dd_{z_i} + 2 \dd_{\eta_i},
\end{align*}
where $i \in \{1, 2\}$, and $z=(z_1, z_2) \in \RR^2, \eta = (\eta_1, \eta_2) \in \RR^2$. We can now check that 
\[\nu .\i_{V_1 - V_2} d\bar \l_T = (\eta_1-\eta_2)\eta\cdot dz,\] hence $X_T$ is directed by $V_1 - V_2$ according to the Moser equation. 
\end{proof}

Hence Moser's trick's equations can be written
\begin{align*} 
    \dd_T (\ln \bar \mu_T) \circ \psi_T^{-1} & = \dd_T \bar \l_T (R_{\bar \l_T})  = - \dd_T \vphi_T'. \nu du (R_T) \\
     & = B \nu^2 \dd_T \vphi_T' = \frac{\nu^2 \dd_T \vphi'_T}{(C+u)^2 \vphi'_T } \\
  (\i_{X_T} d\bar \l_T)_{\rest \xi_T}&  = - \left(\frac{\nu}{C+u}\right)^2 \frac{\dd_T \vphi_T'}{\vphi_T'} \eta \cdot dz_{\rest \xi_T}. 
\end{align*}

{\it {Step 3} : Bounding $\bar \mu_T$} \\
Using the Moser equations, 
$\dd_T (\ln \bar \mu_T) \geq 0$ hence $\bar \mu_T \geq 1$ (recall that $\bar \mu_1 = 1$). Moreover we have
\begin{align*}
    \dd_T (\ln \bar \mu_T) & = \dd_T  \bar \l_T (R_{\bar \l_T}) \circ \psi_T \\
    & \leq \frac{ \dd_T \vphi'_T}{(C+u)^2 \vphi'_T },
\end{align*}
hence
\begin{align*}
    \ln \bar \mu_T  & \leq  \frac{1}{(C+u)^2} \int_1^T \frac{\dd_T \vphi'_T}{\vphi'_T } \quad \text{ using }\ \bar \mu_1 = 1  \\
     & \leq \ln \vphi_T' - \ln \vphi_1' = \ln \vphi_T' \quad {\text using }\ \vphi_1' = 1,
\end{align*}
which implies $1 \leq \bar \mu_T \leq \vphi_T'$. \\~\\

{\it {Step 4} : Bounding $\mu_T$} \\
We now show that $\mu_T = \frac{\bar \mu_T}{\psi_T^* F_T} = \frac{\bar \mu_T}{F_T} $ is bounded as desired. For that purpose we need to choose $\vphi_T$ on $[0 ,\frac16]$ satisfying some conditions.

\begin{fact} \label{fact.technical_function}
    For any $T \geq 1$, there exists $\vphi_T : [0, \frac16] \ra \RR$ strictly increasing and such that
    \begin{itemize}
        \item $\vphi_T(0) = 0$, $\vphi_T'(0) = 1$ and $\vphi_T'(\frac16) = T$;
        \item $\frac{1 + \vphi_T}{\vphi_T'}> \e >0$, where $\e$ is independent of $T$ and $u$;
        \item $\vphi_T(\frac16) \geq \e' T$, where $\e'$ is independent of $T$.
    \end{itemize}
\end{fact}

We delay the proof of this statement, and first prove \cref{eqn.estim_mu_T}, using the previous estimation $1 \leq \bar \mu_T \leq \vphi'_T$, as well as the properties of $\vphi$ described \cref{ssec.forms}. 

Remember that $\mu_T = \frac{\bar \mu_T}{F_T}$, and that for $0 \leq u \leq \frac46$ we have 
\[F_T = (C_0+ \vphi_T(u))(C_0+u).\]
Since we chose $C_0 \geq 1$ (and $T \geq 1$), we get 
\[1+\vphi_T \leq F_T \leq C T,\]
hence 
\[ \mu_T = \frac{\bar \mu_T}{F_T} \geq \frac{\bar \mu_T}{CT}  \geq \frac1{CT} \]
which prove the left side of \cref{eqn.estim_mu_T} for $u \in [0, \frac46]$.
On the other hand, we also have, still for $0 \leq u \leq \frac46$ :
\[ \mu_T = \frac{\bar \mu_T}{F_T} \leq \frac{\bar \mu_T}{1+\vphi_T} \leq \frac{\vphi'_T}{1+\vphi_T} \]
We now argue for each interval :
\begin{enumerate}
    \item For $0 \leq u \leq \frac16$, $\frac{\vphi'_T}{1+\vphi_T}$ is bounded by \cref{fact.technical_function}.

    \item For $\frac16 \leq u \leq \frac26$, $\vphi_T$ is affine, with $\vphi_T' = T$, and by \cref{fact.technical_function}
    \[\vphi_T \geq \vphi_T(\frac16) \geq \e'T, \] 
    hence $\mu_T \leq \frac T{1+ \e'T} \leq \frac1{\e'}$.  

    \item For $\frac26 \leq u \leq  \frac36$, $\vphi_T \geq \frac T6$ and $\vphi'_T \leq T$, hence 
    \[\mu_T \leq \frac{\vphi'_T}{1+\vphi_T} \leq \frac{T}{1+\frac T6} \leq \frac 16 .\] 

    \item For $\frac 36 \leq u \leq 1$, we have $\phi_T' = 1$ so the contact structure is unchanged : $\l_T = \frac{F_1}{F_T}\l_1$.
    Hence $\psi_T$ is the identity, and $\mu_T = \frac{F_1}{F_T}$. Moreover we can chose $F_T$ such that 
    \[ C \leq F_T \leq C T,\]
    which conclude the proof of \cref{eqn.estim_mu_T}.
\end{enumerate}
~\\

\begin{proof}
We now construct a function $\varphi_T$ as wished for \cref{fact.technical_function}. For ease of read, we will actually define $f : [0,1] \ra \RR$. We set
\[ f(x) = \frac13\big(1+ 2T + 3 T (x-1) - (T-1) (x-1)^3\big) \]
Then 
\begin{align*}
    & f(1)   = \frac{2T+1}3 \geq \frac23 T \\
    & f' = T - (T-1)(x-1)^2 \geq 1.
\end{align*}
We can check that $f(0) = 0$, $f'(0) = 1$ and $f'(1) = T$. Finaly, we compute
\begin{align*}
    3 \frac{1+f}{f'} & = \frac{4-T+3Tx-(T-1)(x-1)^3}{(x-1)^2 + T(2x-x^2)  }\\
     & = \frac{4- 3x^2+x^3 + \frac1T (4+(1-x)^3)  }{2x(1-x) + \frac1T (x-1)^2} \\
    & \geq \frac{4- 3x^2+x^3  }{2x(1-x) + 1} \quad \text{ since } \frac{(x-1)^2}{T} \leq 1 \\
    & \geq  \frac{4-3}{2+1} = \frac13 \quad \text{ since } x(1-x) \leq 1,
\end{align*}
\end{proof}

{\it {\bf Step 5}: Divergence of $Q_A$}\\ 
Applying \cref{lem.length_geods} and using the previous estimations, we get \cref{eqn.diverg_Q_A} :
\[Q(A) = \frac A {\mu_{T(A)}} \geq CAT(A) \geq CAe^A\]

which conclude the proof of \cref{lem.bound_mu_T}.
\end{proof}

We can now apply \cref{lem.ColimFormes} to the contact forms $\tilde \l_T$, where the continuation morphisms are inclusions $\ZZ[\pi_1^A(M\sms K)] \subset \ZZ[\pi_1^{A'}(M\sms K)]$, which conclude the proof of the first part of \cref{thm.LH_U(M-K)=gp_ring}. \\

\begin{rmk}
Similarly, the non-linearized contact homology is the algebra freely generated by the $c_\g$. A priori the colimit of dgas is not well-defined (although its homotopy colimit is), however in this case the continuation morphisms are simply inclusions of dga. \\
Regarding the contact homology the computation is more difficult : we now have one Reeb orbit by free homotopy class, so for the cylindrical contact homology the differential trivially vanishes. However if we allow for several negative asymptotics, the previous action argument doesn't hold. 
\end{rmk}

\begin{rmk} \label{rmk.BKO}
Another way to define the contact homology of the unit bundle of $M \sms K$ was described in \cite{BKO19_Formal_boundary_hyperb_knot_compl}. Surprisingly (at least for the author), one can actually control holomorphic curves when using the contact form induced by the hyperbolic metric. Let us summarise their construction from the sutured point of view. \\
Outside of a compact, $M\sms K \simeq [1,\infty)_u \times \TT^2$, with the metric $(du^2+g_0)/u^2$, where $g_0$ is the flat metric on $\TT^2$, and the induced contact form is "concave". Indeed if we consider the unit bundle of $\{u\leq C\}$ (which is a manifold with boundary), it has convex boundary, and near the dividing set the Reeb vector field goes from the positive to the negative region (as depicted on the left side of \cref{fig.colim_metriq}). \\
However, by lifting $M\sms K$ to the standard hyperbolic space $\HH^3 = (0,\infty)_z \times \RR^2$, the metric becomes $(dz^2+g_0)/z^2$ (note that $u=\infty$ correspond to $z=0$). And now the induced contact form is convex : if we consider the unit bundle of $\{z\leq C\}$, the Reeb vector field is going from the negative to the positive region, so it is adapted to the sutured contact manifold. \\
This is not a contradiction, and the reason of that discrepancy is that $z$-levels and $u$-levels don't coincide. It also illustrates that, when working with non-compact sutured manifolds, it is crucial to choose coordinates at infinity.
Moreover, the previous colimit argument shows that the two definitions coincide (for Legendrian homology) : one can use either the sutured setting, or the contact form induced by the hyperbolic metric (in which case the almost complex structure has to be the Sasakian one, associated to the metric).
\end{rmk}

\subsection{Product structure.} \label{ssec.prod}

According to the previous section, to compute the product on $LH(\La_x, \tilde V, \tilde \l_h)$ it is enough to determines curves for the contact form induced by the hyperbolic metric $g_h$ on $M\sms K$. Lifting the point $x$ to the hyperbolic universal cover $\HH^3$, we obtain a lattice $Q \subset \HH^3$ of conormal \[\La_Q = \{(x, a), x \in Q, a \in U^*_x \HH^3\} \subset U \HH^3\] here endowed with the contact form coming from the hyperbolic metric. 
Then for any pair of points $x, y \in Q$, there is an unique hyperbolic geodesic $c_{xy}$ between them, lifting to a Reeb chord in the unit bundle $U \HH^3$, with ends in $\La_Q$.

\begin{lem}\label{lem.hyperb.produit}
For any distinct $x, y, z \in Q$, there exists a unique holomorphic curve in the moduli space $\M(c_{xy}c_{yz}; c_{xz})$.
\end{lem}

This result conludes the proof of \cref{thm.LH_U(M-K)=gp_ring} : for any $\g, \g' \in \pi_1(M\sms K)$, such that $\g' \neq \g^{-1}$, and $T$ big enough, there is a unique holomorphic pair of pants in the symplectisation of $(\tilde V, \l_T)$, with boundary in $\RR \times \La_x$, and positively asymptotic to $c_\g$ and $c_{\g'}$. Moreover the curves is negatively asymptotic to $c_{\g \g'}$ (which is the only possible negative chord by homotopy reasons).

\begin{proof}
The strategy is as follows : we first lift $M \sms K$ to the standard hyperbolic space, whose unit bundle is contactomorphic to $J^1(S^2)$ with the standard contact structure. The contact form induced by the hyperbolic metric differs from the standard form by a conformal factor, so we interpolate between them (note that the standard contact form comes from the euclidian metric on $\RR^3$). After this procedure, we can use the Morse flow trees to explicit the holomorphic curves. \\

{\it (i) Euclidian metric} : \\We use the invariance of the product to work with the contact form coming from the euclidian metric on $\RR^3$. Let $g^s$ be a path of metrics on $\RR^3$, interpolating between the hyperbolic metric $g_h$ and the flat metric $g_\eu$. It induces an exact Lagrangian cobordism $L \subset (W, \b)$, between $\La_Q \subset (U_{g_h} \RR^3, \l_{g_h})$ and $\La_Q \subset (U_{g_\eu} \RR^3, \l_{g_\st})$. 

For this last form, there also exists a unique Reeb chord joining two fibers. The map obtained by counting holomorphic strips in $W$, with boundary in $L$, is a quasi-isomorphism hence it is trivial at the chain level (curves with negatively asymptotic to several curves aren't taken into account, since there is no contractible Reeb chord). \\~\\

{\it (ii) 1-jets space} : \\
We now use the usual exact contactomorphism between the euclidian unit bundle $(U_{g_\eu} \RR^3, \l_\eu)$ and the $1$-jet space $(J^1(S^2), \l_\st)$, with coordinates as follows
\[J^1(S^2) = \{(s, q, p) \in \RR\times  \RR^3 \times T_q \RR^3\ |\ |q|^2 =1, q \cdot p =0\}.\] 
The map now takes the form 
\begin{align*}
     & J^1(S^2) \ra U \RR^3 = \{(\xx; v) \in \RR^3 \times T_z \RR^3, v^2 = 1 \} \\
     \phi : \ &  (t, q, p) \mt (tq + p  , q)
\end{align*}
and indeed $\phi^* \l_\eu = ds - pdq$.
The fiber $U_0 \RR^3$ becomes the zero section, and more generally the fiber $U_x \RR^3$ becomes the $1$-jet of the function $f_x : q \mt x \cdot q$.
Indeed we have $\phi^{-1}(U_x \RR^3) = \{ sq + p = x \} \subset J^1(S^2)$, and $p$ being orthogonal to $q$ we get
\begin{align*}
    s & =  x \cdot q \\
    p & =  x - sq = x - (x \cdot q) q \\
     & = d(x \cdot q).
\end{align*}
~\\

{\it (iii) Morse flow trees} : \\
In a $1$-jets space with the standard contact form, holomorphic curves are in correspondence with the Morse flow trees as proved in \cite{Ek05_Morse_flow_trees}. We give a brief review of the relevant part of this construction.

Given three functions $f_0, f_1, f_2 : S^2 \ra \RR$, Reeb chords going from $j^1 (f_i)$ à $j^1(f_j)$ (of Maslov degree $k$) correspond to positive critical points of $f_j - f_i$ (of Morse degree $k +1$). For all $0\leq i < j \leq 2$, pick a cord $c_{ij} \in \C(j^1(f_i), j^1(f_j))$, corresponding to a critical point $q_{ij} \in S^2$. 
Then a Morse flow trees of asymptotics $(q_{01}, q_{12}; q_{02})$ will be a collection of paths $p_{01}, p_{12}, p_{02} \subset S^2$ such that 
\begin{itemize}
    \item $p_{01}$ is a trajectory of $- \nabla (f_1 - f_0)$, going from $q_{01}$ to $q_* \in S^2$ ;
    \item $p_{12}$ is a trajectory of $- \nabla (f_2 - f_1)$, going from $q_{12}$ to $q_*$ ;
    \item $p_{02}$ is a trajectory of $- \nabla (f_2 - f_0)$, going from $q_*$ to $q_{02}$.
\end{itemize}
Note that a path can be constant, in which case $q_*$ coincides with one of the critical points.
If $|q_{01}| + |q_{12}| - |q_{02}| = 2$ (ie if $|c_{01}| + |c_{12}| - |c_{02}| = 1$), those trees will be rigid, and the result from \cite{Ek05_Morse_flow_trees} affirm that those trees are in correspondence with rigid holomorphic pair of pants. \\

{\it (iv) Product in $J^1(S^2)$} : \\
The function $f_y - f_x : q \mt q \cdot (y - x)$ presents two  critical points, $q_{xy} = \frac{y-x}{|y - x|}$ and its opposite, corresponding to Reeb chords projecting to the geodesic between $x$ and $y$. Denote $c_{xy}$ the chord going from $\La_x$ to $\La_y$, corresponding to $q_{xy}$. Then for any triplet of points $x, y, z \in \RR^3$, there exists a unique rigide Morse flow tree asymptotic to $(q_{xy}, q_{yz}; q_{xz})$.

Indeed $q_{xz}$ is a maximum de $f_z - f_x$, hence such a Morse flow tree must a constant negative branch (ie $q_* = q_{xz}$). Moreover, there is a unique trajectory of $- \nabla(f_y-f_x)$ going from the maximum $q_{xy}$ to $q_{xz}$, et similarly for $q_{yz}$. 
\end{proof}

\begin{rmk}
Note that $U\RR^3$ is non compact, however the curves stay bounded (for fixed asymptotics) since it can be seen as the completion of a sutured contact manifold. 
Alternatively we could use the maximum principle from \cite[Theorem 8.4]{BKO19_Formal_boundary_hyperb_knot_compl}. 
\end{rmk}

\section{The complement of the conormal of the knot} \label{sec.UM-UK}

We first describe the gluing of a circular contact handle $S^1 \times D^2 \times D^2$ along the horizontal boundary $R_-$, and prove that sutured Legendrian homology is unchanged during this operation. This result stems from a 1-parameter family of contact manifold constructed in \cite{CGHH10_Sutures}, stretching the gluing locus.

We finally show that the contact manifold $UM \sms \N(\La_K)$ can be obtained from $U(M \sms \N(K))$ by gluing two of those handles, concluding the proof of our main theorem. The handles being glued along (a submanifold of) the suture of the previous section, we will have to perturb the contact form, displacing the suture to which it is adapted.

\subsection{Circular handles} \label{ssec.Circ_handle}

\paragraph{Attaching handles.}
Consider the $5$-dimensional handle 
\[H = \{ (s, x, y) \in S^1 \times \RR^2 \times \RR^2, |x|^2 \leq 1, |y|^2 \leq 1 \}, \]
endowed with the contact form $\lambda = (1+\e x \cdot y) ds - y dx$, with $\e \ll 1$. 
The Reeb vector field is directed by
\[R \sim \dd_s + \e(x \dd_x - y \dd_y),\]
whose projection to the $\RR^4$-factor is the Hamiltonian vector field of $\e x \cdot y$, see also \cref{lem.Reeb.S1xW} for a more general formula.
In particular, this contact handle have one Reeb orbit contained in $\{x=y=0\}$, whose Conley-Zehnder is $0$.

\begin{rmk}
There is no specific reason to choose the Liouville form $-y.dx$, instead of $xdy$ or $xdy + ydx$. Modifying this part of the contact form slightly change the following formulas, but the results still hold.
Also note that we could take $\e = 0$, and instead perturb the ball in $\RR^4$ into an ellipsoid.
\end{rmk}

\begin{figure}[h!]
\center
\def\svgwidth{10cm} 
\begingroup%
  \makeatletter%
  \providecommand\color[2][]{%
    \errmessage{(Inkscape) Color is used for the text in Inkscape, but the package 'color.sty' is not loaded}%
    \renewcommand\color[2][]{}%
  }%
  \providecommand\transparent[1]{%
    \errmessage{(Inkscape) Transparency is used (non-zero) for the text in Inkscape, but the package 'transparent.sty' is not loaded}%
    \renewcommand\transparent[1]{}%
  }%
  \providecommand\rotatebox[2]{#2}%
  \newcommand*\fsize{\dimexpr\f@size pt\relax}%
  \newcommand*\lineheight[1]{\fontsize{\fsize}{#1\fsize}\selectfont}%
  \ifx\svgwidth\undefined%
    \setlength{\unitlength}{314.03052142bp}%
    \ifx\svgscale\undefined%
      \relax%
    \else%
      \setlength{\unitlength}{\unitlength * \real{\svgscale}}%
    \fi%
  \else%
    \setlength{\unitlength}{\svgwidth}%
  \fi%
  \global\let\svgwidth\undefined%
  \global\let\svgscale\undefined%
  \makeatother%
  \begin{picture}(1,0.95003942)%
    \lineheight{1}%
    \setlength\tabcolsep{0pt}%
    \put(0,0){\includegraphics[width=\unitlength,page=1]{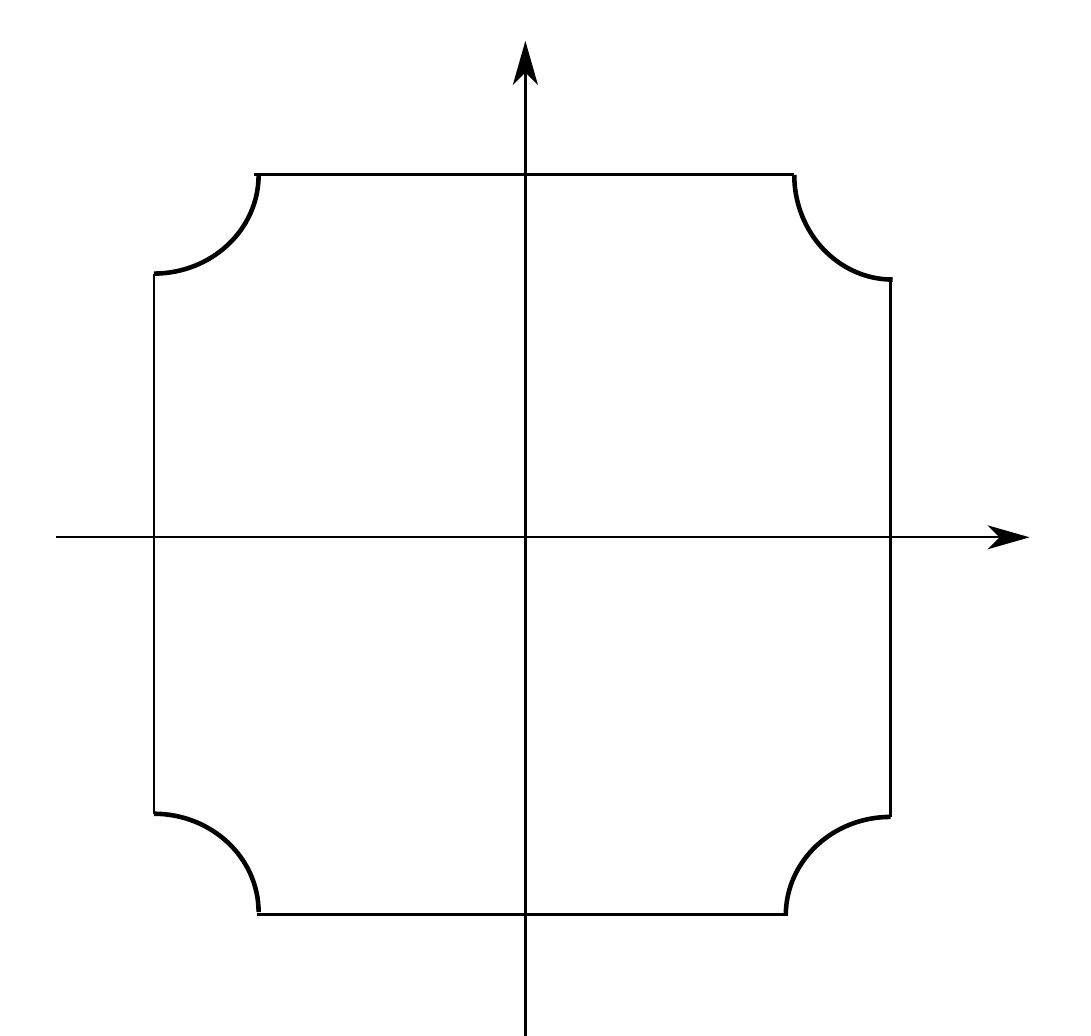}}%
    \put(0.89790005,0.49222451){\color[rgb]{0,0,0}\makebox(0,0)[lt]{\lineheight{1.25}\smash{\begin{tabular}[t]{l}$x$\end{tabular}}}}%
    \put(0.51930187,0.90649021){\color[rgb]{0,0,0}\makebox(0,0)[lt]{\lineheight{1.25}\smash{\begin{tabular}[t]{l}$y$\end{tabular}}}}%
    \put(0.31819582,0.04062089){\color[rgb]{0,0,0}\makebox(0,0)[lt]{\lineheight{1.25}\smash{\begin{tabular}[t]{l}$\dd_- H$\end{tabular}}}}%
    \put(-0.00475795,0.50675536){\color[rgb]{0,0,0}\makebox(0,0)[lt]{\lineheight{1.25}\smash{\begin{tabular}[t]{l}$\dd_+ H$\end{tabular}}}}%
    \put(0,0){\includegraphics[width=\unitlength,page=2]{Anse_hyperb.pdf}}%
  \end{picture}%
\endgroup%

\caption{The Reeb vector field (projected) on a circular hyperbolic handle. The dotted lines bound the gluing locus.}
\label{fig.anse_hyp}
\end{figure}

After removing a neighbourhood of the corners, such that the Reeb vector field is tangent to the newly created boundary, we get a manifold with with corners. Moreover the Reeb vector is negatively transverse to $\dd_- H = S^1 \times D^2  \times S^1$, and negatively transverse to $\dd_+ H = S^1 \times S^1  \times D^2$, see \cref{fig.detach_handle}. 
However, it is {\it not} a sutured contact manifold, because the Liouville vector field is not outgoing everywhere. Regardless, we can still glue it to another sutured contact manifold $(V, \l)$, provided we have an exact symplectomorphism $\phi : \dd_- \tilde{H} \hra R_+(V)$.
We obtain that way a new sutured manifold $\tilde{V} = V \under\cup\phi H$, whose dividing set is unchanged.

\paragraph{Detaching handles.}
We now give a simple way to find an handle in a contact manifold $(V, \xi)$, with convex boundary. Choose a contact form $\l$ adapted to the boundary, of associated dividing set $\Ga$, and consider a convex hypersurface $\Si \subset (V,\xi)$ such that :
\begin{itemize}
    \item $\Si = \Si_\dd \cup \Si_0$, where $\Si_\dd \subset \dd V$ and $\Si_0 \subset \mr V$ ;
    \item $\l$ is adapted to $\Si$, with associated dividing set $\Ga_\Si$ ;
    \item $\Ga  = \Ga_\Si  \subset \Si_\dd$ ;
\end{itemize}
We refer to \cref{fig.detach_handle} for a depiction of the situation.

\begin{figure}[h!]
\center
\def\svgwidth{12.5cm}
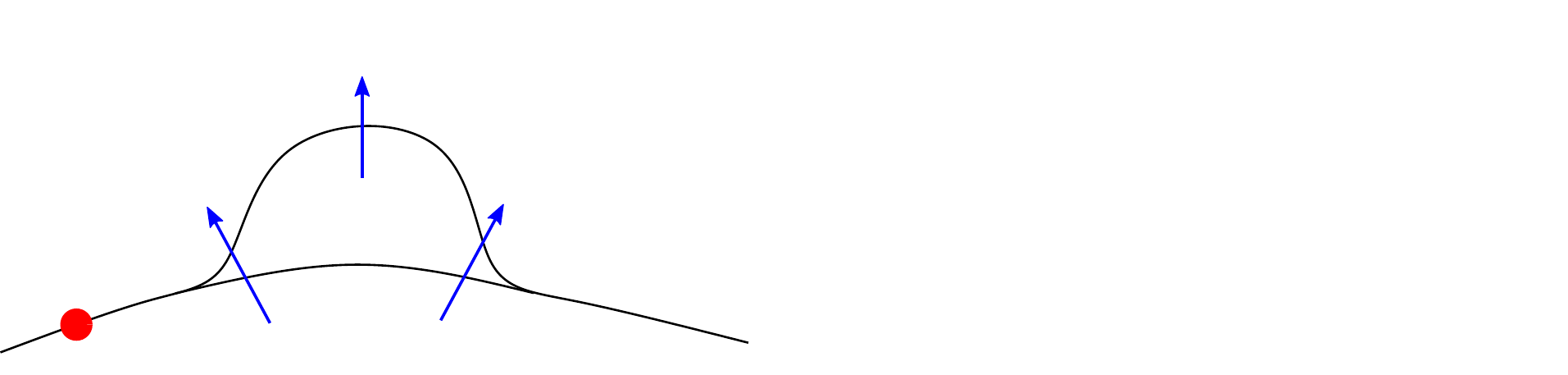
\caption{Detaching an handle. $\Ga$ is the dividing sets of both $\dd V$ and $\dd V_0 = \Si$.}
\label{fig.detach_handle}
\end{figure}

Denote by
\begin{itemize}
    \item $V_0$ the manifold obtained by removing the part between $\Si_0$ and $V \sms \Si_\dd$, with convex boundary $\dd V_0 = \Si$ ;
    \item $H$ the sutured contact manifold 
    \[H = I\times (R_+ \sqcup R_-) \cup V\sms V_0. \]
    This manifold has suture $\Ga$, and the horizontal boundaries are $\Si$ and $\dd V$.
\end{itemize}

Then $V$ can be reconstructed by gluing the sutured manifold $(H, \l)$ to $V_0$, by a symplectomorphism $R_\pm(H) \ra R_\mp(V_0)$ (which is the identity on the boundary), as described in \cite[§8]{CGHH10_Sutures} (note that this is different from the gluing in \cite[§4.3]{CGHH10_Sutures}).

\begin{thm}\label{thm.detach_handle} Consider a convex hypersurface $\Si \subset (V,\l)$ as above. 
Then 
\[CC(V, \l) \simeq CC(V_0, \l) \otimes CC(H, \l). \]
Moreover if $\La \subset V_0$ is a Legendrian\footnote{Here we assume that neither $V_0$ or $H$ has a contractible Reeb orbit. If not, the Legendrian homology will a module over the contact homology generated by contractible chords.}, we have
\[ LC(\La; V, \l) \simeq LC(\La; V_0, \l) \]
\end{thm}

\begin{proof}
It is a direct consequence of \cite[Theorem 8.1]{CGHH10_Sutures} : by hypothesis, Reeb orbits and Reeb chords don't intersect the gluing locus $\Si_0$.
\end{proof}

For the circular handles described in the previous section, we can compute the sutured contact homology of the handles by a simple action argument :
\begin{cor}
If $\tilde V$ is obtained from $V$ by gluing a circular handle, and that the map $H_1(\TT^2) \ra H_1(R_+(V))$ induced by the gluing is injective. Then the contact homologies are related by 
\[CH(\tilde V) \simeq CH(V) \otimes \ZZ[(c_k)_{k \in \NN \sms 0}].\]
\end{cor}

\begin{proof}
The only orbit in the circular handle is the core $S^1 \times \{0,0\}$, which has Conley-Zehnder index $0$. Denote by $c_k$ the $k$-cover of this orbit. By homotopy reasons, the differential of $c_k$ can only have terms of the form $c_{a_1}...c_{a_l}$ such that $\sum a_i = k$. However, by action reason the curve has vanishing energy, so it is contained in the cylinder over the (simple) orbit. And according to \cite{Fab07_Obstruction_bundles_Action_filtration}, multiples covers of trivial holomorphic cylinders don't contribute to the differential.
\end{proof}

\subsection{From the unit bundle of the complement, to the complement of the conormal} \label{ssec.final}

In this section the relate the complement of the conormal $\La_K \subset U^* M$ to the unit bundle of the complement of $K \subset M$.

\begin{prop} \label{lem.U(-).-(U)} There exists a standard neighbourhood of $\La_K \subset UM$, and a contact form $\l$ adapted to the convex boundary, such that $(U M \sms \N(\La_K), \l)$ is obtained from $U(M \sms \N(K))$, endowed with an adapted contact form, by gluing to circular handles $(S^1 \times D^2 \times D^2, \l_H)$ along $R_-$. 
\end{prop}

Topologically, $UM \sms \La_K$ is indeed obtained from $UM \sms \La_K$ (here we omits the neighbourhood for ease of read) by gluing two handles $S^1 \times D^2 \times D^2$ along two tori (which are, in the coordinates defined below, $\{z^2=\d, \t =\pm 1, \eta=0\}$). However those tori are included in the suture associated to a contact form induced by a metric (cylindrical near the boundary). We will change the contact form so that is adapted to a suture disjoint from the gluing loci, to recover a situation to which \cref{thm.detach_handle} applies. We will then have to check that the resulting contact manifold is indeed the complement of a standard neighbourhood of $\La_K$.

The full procedure is illustrated \cref{fig.hyperb.recollmt_anse}, on the left in stereographical charts (relatively to the fiber), and on the right as a sphere bundle. The Reeb vector field is depicted in blue, and the suture in red. In charts, the Reeb vector field is dotted when it is directed by the the $S^1$ factor, and the suture is depicted by one or two points (resp. an hollowed point) when a whole circle (resp. only one point) of the fiber is part of the suture.

\begin{figure}[h!]
\center
\def\svgwidth{12.2cm}
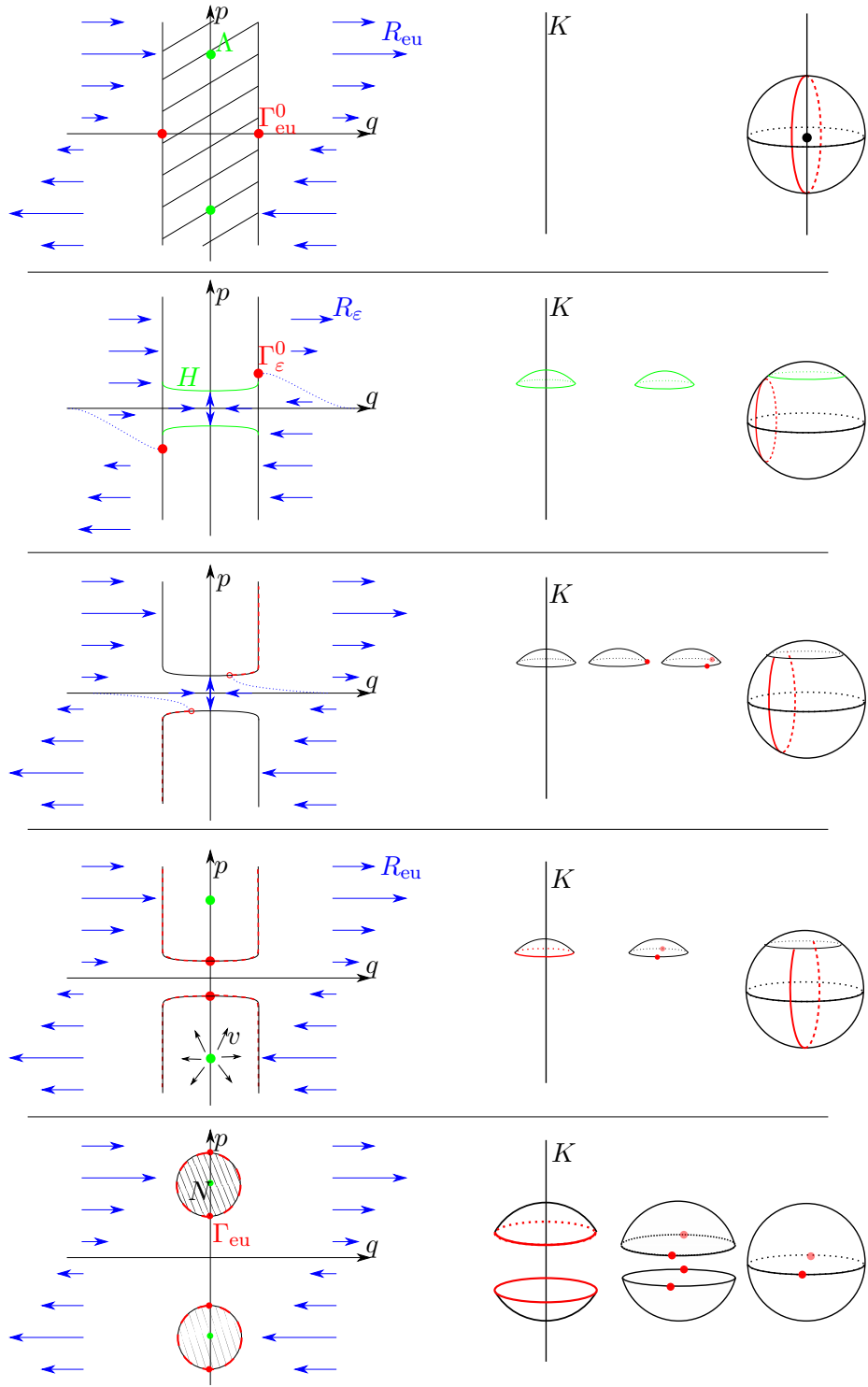
\caption{Gluing an handle, on the left in stereographical charts, and on the right in a sphere bundle. See \cref{ssec.final} for a more precise description.}
\label{fig.hyperb.recollmt_anse} 
\end{figure}

Line by line are depicted :
\begin{itemize}
    \item the contact manifold $(V_0, \l_\eu)$, with its convex boundary ;
    \item the contact manifold $(V_0, \l_\e)$ with a glued handle ; 
    \item the perturbed form $(V_1, \l_{\e'})$ for $\e' << \e$ ; 
    \item the contact manifold $(V_1, \l_\eu)$ and a contact vector field transverse to the boundary ; 
    \item the contact manifold $(V, \l_\eu)$, with its convex boundary.
\end{itemize}

\subsubsection{Stereographic charts} \label{sssec.Stereo_charts}

Consider the unit tangent bundle coming from the euclidian metric on a neighbourhood of the knot. Starting with coordinates $(t, z) \in S^1 \times D^2$ near the knot, such that $K = S^1 \times \{0\}$, we denote by $UM \subset TM$ the manifold obtained by taking the unit fiber relatively to the metric $dt^2 + dz^2$, coming with coordinates $(\t,\eta) \in \RR \times \CC$ on the fiber : 
\begin{align*}
     TM \supset\ & U(S^1 \times D^2) \simeq S^1_t \times D^2_z \times S^2_{\t, \eta} \\
      & = \{(t, z; \t, \eta) \in \RR/\ZZ \times \CC \times \RR \times \CC \ |\  \t^2 + |\eta|^2 =1 \}.  
\end{align*}
The conormal of the knot becomes $\La = \{z= 0, \t = 0\}$, and the contact form induced by this metric is
$\l_\eu = -\t dt - \eta\cdot dz$.
Moreover the Reeb vector field is 
\begin{equation} \label{eqn.Hyperb.Reeb_eu}
    R_\eu = \t \dd_t + \eta . \dd_z.
\end{equation} 

\paragraph{Standard neighbourhood.}
We now define a standard neighbourhood of the Legendrian, which will have convex boundary :
\[ N = \{(t, z; \t, \eta) | \t^2+ |\eta|^2=1, \t^2+ |z|^2 < \d  \} \subset UM,  \]
and we set $V = UM \sms N$. 
Then for $\d$ small enough, $N$ is a perturbation of a standard neighbourhood of $\La$, with convex boundary, and the contact form $\l_\eu$ is adapted to the suture 
\[ \Ga_\eu = \{ (t, z; \t, \eta)\ | \t^2 + \eta^2 = 1, \t^2 + |z|^2 = \d,  \eta \cdot z = 0 \}. \]

\paragraph{Stereographic charts.}
We will use two charts to reduces the situation to an simpler manifold, given by the stereographic projection of the fiber relatively to the point $(\t, \eta)=(1, 0)$ :
\begin{align*}
 & S^1 \times T^* D^2 \lra U \subset UM\\
 \vphi_\pm : &\ (s, q, p) \mt (t = s, z=q, \t = \pm \frac{1- |p|^2}{1 + |p|^2}, \eta = \frac{2p}{1+ |p|^2})
\end{align*}
The projection to the base is the map forgetting the coordinate $p$,  and the Legendrian now becomes
\[ \vphi_\pm^*(\La) = \{q=0, |p|=1\}.\]

The inverse of those maps are $\vphi_\pm^{-1} : (t, z; \t, \eta) \mt (s=t, q=z, p=\frac{\eta}{1 \pm \t})$, and the change of charts is given by $(s, q, p) \mt (s, q,1/p)$ (here $p$ is seen as a complex).

We now compute  
\begin{equation}\label{eqn.hyperb.forme_eu.cartes}
  \vphi_\pm^* \l_\eu = \mp \frac{1- |p|^2}{1 + |p|^2}ds - \frac{2}{1+|p|^2} p.dq  
\end{equation}
In those charts the conormal is $\vphi_\pm^{-1}(\La) = \{q=0, |p|=1 \}$, and the Reeb vector field becomes 
\begin{equation} \label{eqn.hyperb.Reeb_eu.cartes}
    \vphi_\pm^* R_\eu  
     = \pm \frac{1- |p|^2}{1 + |p|^2}  \dd_s + \frac{2}{1 + |p|^2} p.\dd_q  
\end{equation}

Moreover $\vphi_\pm^{-1} (N) = \{ q^2 + (\frac{1-p^2}{1+p^2})^2 < \d \}$ and  $\vphi_\pm^{-1} (\Ga_\eu) = \{p \cdot q = 0\}$. 

\subsubsection{Modifying the contact form} \label{sssec.Modify_contact_form}

We construct a contact form on $V = UM \sms N$, which can be recovered from $V_0$, unit bundle of the complement of the knot, by gluing two circular handles on the horizontal boundary. 
Recall that, topologically, we want to glue those handles along the tori $\{|z| = \d, \t = \pm 1, \eta = 0 \}$, however they are part of the suture. We will now perturb the contact form so that it is adapted to a suture disjoint from those tori, so we will be able to apply the gluing result \cref{thm.detach_handle}.

\begin{rmk}[Prototype]
In a $1$-jets space, endowed with the standard contact form (which is adapted to the suture associated to the radial contact vector field), we can displace the suture away from the zero section. Indeed the Reeb vector field is $\dd_s$, hence  $X = (s+ \e)\dd_s + p\dd_p$ is a contact vector field transverse to the hypersurface $\{ s^2 + p^2 = \d \}$, 
and the suture induced is then $\Ga_X = \{X \in \xi\} = \{s^2 + p^2=\d, s = -\e \}$.
We will use a similar construction to displace the suture to which $\l_\eu$ is adapted.
\end{rmk}

We first compute the Reeb vector field of the contact form $\lambda = \phi(x)(f(x) ds + \b)$ on $S^1_s \times W_x$, where $(W, \b)$ is a Liouville manifold.

\begin{lem} \label{lem.Reeb.S1xW} The contact condition is $ds \we (fd\b - df \we \b) \we(d\b)^{n-1} > 0$,
and the Reeb vector field is
\[R = \frac{1}{\phi^2}\left( \frac{\phi + d\phi(Y)}{f - df(Y)} (\dd_s+X_f) +  \frac{ df(X_\phi)}{f -df(Y)} Y + X_\phi\right)\]
where $Y$ is the Liouville vector field associated to $\b$ and $X_f$ is defined by $\i_{X_f} d\b = df$.
\end{lem}

In particular the projection of the Reeb vector field, parallel to the $S^1$ factor, is directed by 
\[ \pi(R) \sim  \big(\phi + d\phi(Y)\big) X_f +  \big(f - df(Y)\big) X_\phi + df(X_\phi) Y. \]

\begin{proof}
The contact condition is a simple computation. For the formula of the Reeb vector field, we start by computing the Reeb vector field of $\l_0 = f.ds + d\b$. We get 
\[R_0 = \frac{1}{f + \b(X_f)}(\dd_ s + X_f).\]

By a quick computation, the Reeb vector field of $\l$ is $R = \frac{1}{\phi^2}(\phi R_0 + V)$, where $V \in \xi$ satisfies $(\i_V d\l_0)_{\rest \xi} = d\phi_{\rest \xi}$. We split this vector into $V =  \b(v) \dd_s - f v$, where $v \in TW$, and we now compute $v$. 
We have
\begin{align*}
    \i_{X_\phi}d\b_{\rest \xi} & =   d\phi_{\rest \xi} = \i_{\b(v) \dd_s - f v}(df \wedge ds + d\b) \\
     & = -f df(v) ds - \b(v) df - f \i_v d\b \\
     & = df(v)\b  - \b(v) df - f \i_v d\b \quad \text{ since } \xi = \ker fds + \b \\
     & = \i_{df(v) Y -\b(v) X_f-f v} d\b     
\end{align*}
Hence $fv = df(v) Y - \b(v) X_f - X_\phi$. Applying $\b$ and $df$ we obtain
\begin{align*}
    & (f+\b(X_f))\b(v) = - \b(X_\phi) \\
    & (f- df(Y))df(v) = - df(X_\phi),
\end{align*}
and using the identity $\b(X_f) = 
- df(Y)$, the vector $V$ can be written
\begin{align*}
    V & = \b(v) \dd_s - f v \\
     & = \b(v) (\dd_s +X_f) - df(v) Y +X_\phi \\
     & = \frac{1}{f+ \b(X_f)} \big(  -\b(X_\phi) (\dd_s +X_f) +df(X_\phi) Y  \big)  +  X_\phi
\end{align*}

Finally the Reeb vector field associated to $\l$ is  
\begin{align*}
R & = \frac{1}{\phi^2}(\phi R_0 + V) \\
  & = \frac{1}{\phi^2} \left( \frac{\phi - \b(X_\phi)}{f + \b(X_f)}(\dd_ s + X_f) + \frac{ df(X_\phi)}{f+ \b(X_f)} Y   +  X_\phi \right)
\end{align*}
\end{proof}

In the previously defined coordinates, the unit bundle of the complement of the knot is 
\[ V_0 = \{ q^2 \geq \delta \} \subset UM.\]

\begin{lem}
This manifold has convex boundary, and the contact form $\l_\eu$ is adapted to the suture 
\[ \Ga_0 = \{ (s, q, p), q^2 = \d, p\cdot q =0  \}. \]
\end{lem}
\begin{proof}
Indeed the Reeb vector field is tangent to the boundary when $q \cdot p = 0$, and this set is a contact submanifold : its tangent space is the intersection of the forms $q dq$ and $pdq + qdp$ (which are trivially independent on $\Ga_0$), hence it is generated by the vectors $\dd_s, V = iq \dd_p, W = iq \dd_q + ip \dd_p$. 
We can now check the contact condition :
\begin{align*}
    \l \we d\l(\dd_s, V,W) & =  \big (4pdq\we pdp + 2 (1-p^2) dq\we dp \big)(V,W) \\
     & = - 4 (p \cdot iq)^2 - 2 (1-p^2) (iq \cdot iq) \\
     & = -4 p^2 q^2 - 2(1-p^2) q^2 \\
     & = -2 (p^2 q^2 + q^2).
\end{align*}
Moreover, the Liouville vector field associated to $\l_\eu$ on $\{\pm p\cdot q > 0\}$ is outgoing, because $\Ga_0$ is a connected contact boundary. We can also compute that this vector field is directed by 
\[ Y_\pm \sim (2p - \frac{1+p^2}{p\cdot q} q) \dd_p,  \]
which is indeed transverse to $\{p\cdot q = 0\}$, so \cref{lem.cvx_adapt} applies. 
\end{proof}

\paragraph{Moving the suture.}
We now change the contact form (but not the contact structure) to displace the suture away from the gluing loci. We set
\[ \l_\e = (1 + 2 \e p \cdot q) \big((1-p^2)ds - 2 p dq \big). \]

\begin{prop}
The Reeb vector field associated to $\l$, projected parallel to $\dd_s$, is directed by
\[  \big( \e (1+ p^2)q- (1 + 4\e p\cdot q) p \big) \dd_q - \e (1-p^2)  p\dd_p   \]
and this form is adapted to the convex boundary, with associated suture 
\[ \Ga_\e = \{ p\cdot q = (1+p^2) \d \e + O(\e^2) \}   . \]
\end{prop}

\begin{proof}
For $\e$ small enough, this form is a perturbation of $\l_0 = (1-p^2) ds - 2 pdq$, whose Reeb vector field, projected parallely to $\dd_s$, is directed by $- p \dd_q$. As seen previously, this vector field is tangent to the boundary when $p \cdot q=0$, which determines a contact submanifold. 

The requirements "$\Ga_\e = \{x \in \dd V_0, dq(R_\e)=0\}$ is a contact submanifold" and "the Liouville vector fields are outgoing" being open conditions, the form $\l_\e$ is adapted to the suture $\Ga_\e$ for $\e$ small enough.
Applying the previous lemma with 
\[ \b = - 2 pdq  \quad \quad  f = 1 - p^2 \quad \quad    \phi = 1 + 2 \e p \cdot q, \]
we obtain the desired formula.

We now compute the suture $\Ga_\e$, setting $\chi = q \cdot p$ :
\begin{align*}
    q dq(R_\e) & =  - 4 \e X^2 - X + \e (1+p^2) \d, 
\end{align*}
hence at the first order in $\e$ the elements of $\Ga_\e$ satisfy
\begin{align*}
\chi & = \frac{1}{8\e} (-1 \pm \sqrt{1 + 16 \e^2 (1+p^2) \d  }  )  \\
 & = (1+p^2)\d \e + O(\e^2)
\end{align*}
(only the positive solution survives). 
\end{proof}

\paragraph{Detaching handles.}
The new suture $\Ga_\e$ is disjoint from the torus $\{q^2 = \d, p = 0\}$ (which is now part of the horizontal boundary $R_-$), so we can glue a circular handle : for $\e' << \e$, we define
\[ V_1 = \{q^2 \geq \d\} \cup  \{p^2 \leq \e'\}, \]
where the corners are smoothed.
According to \cref{ssec.Circ_handle}, if the contact form is adapted to the new boundary $\dd V_1$, then for any Legendrian $\La \subset V_0$ we have 
\[ LH(\La, V_0, \l_\e) \simeq LH(\La, V_1, \l_\e).  \]

To effectively get an adapted contact form, we slightly hollow out the handle, meaning that we now glue $\{q^2 \leq \d, p^2 \leq \mu(q^2)\}$), where $\mu : [0, \dd^2] \ra \RR^+$ is increasing, such that $\mu_0 = \e'$ and $\mu'(\dd^2) = +\ity$. 
\begin{lem}
 The contact form $\l_\eu$ is adapted to $V_1$.
\end{lem}

\begin{proof}
After smoothing, the new boundary splits into
\[\dd V_1 = \{ q^2 = \d, p^2 \geq \e' \} \cup \{q^2 \leq \d, p^2 = \mu(q^2) \}. \]
The Reeb vector field $R_\eu$ (projected parallely to $\dd_s$) being directed by $p \dd_q$, it is tangent to the boundary when $X = p \cdot q = 0$, hence the suture splits into  
\[\Ga_1 = \{ q^2 = \d, p^2 \geq \e', X=0 \} \cup \{q^2 \leq \d, p^2 = \mu(q^2), X=0 \}. \]
As previously, the first part is a contact submanifold. The tangent space of the second part is the intersection of the forms $pdq +q dp$ and $pdp - \mu' qdq$, which are independents : if they were proportional, $p$ and $q$ would be colinear, and so $p\cdot q=0$ implies $q= 0$ (since $p$ is non-zero on the new boundary). But $p.dq$ and $p.dp$ aren't proportional.

We now check the contact condition : the tangent space is generated by the vectors $\dd_s$ and
\begin{align*}
    & V = iq \dd_q + ip \dd_p \\
    & W = p^2 ip \dd_q + \mu' (q \cdot ip) p \dd_p
\end{align*}
We now compute :
\begin{align*}
    (df \wedge \b + fd\b) (V, W) & = \big( 4 pdp  \wedge pdq + 2 (1-p^2)dq \wedge dp \big)(V, W) \\
     & = - 4 \mu' (q \cdot ip)p^2 (p \cdot iq) + 2 (1-p^2)\big( \mu' (q \cdot ip) (iq \cdot p) - p^4 \big) \\
     & = - 2 (1-p^2)p^4 -  2 \mu' (q \cdot ip)^2 (1 - p^2 - 2 p^2)\big) \neq 0 
\end{align*}
hence $\l_\eu$ is adapted to the boundary of $V_1$.
\end{proof}

By changing charts and repeating the procedure, we obtain a contact manifold 
\[V_2 = UM \sms \tilde N,\]
with boundary $(\ker \l_\eu)$-convex, and such that $\l_\eu$ is an adapted contact form. A priori $\tilde N$ is distinct from the standard neighbourhood $N$ defined above, however we can interpolates between them through convex hypersurfaces. Indeed the vector field
\[ v = q \dd_q - \frac{1-p^2}{1+p^2} p\dd_p \]
stays transverse to the boundary, and is contact : 
\begin{align*}
    \L_v \l_0 & = \L_{q \dd_q - f p\dd_p} ((1-p^2) ds - 2 pdq)  \quad \text{où } f = \frac{1-p^2}{1+p^2}\\
    & = - d(2 p \cdot q) + \i_v (-2p dp \wedge ds + 2 dq \wedge dp) \\ 
    & = -2 (pdq + qdp) + 2 f p^2 ds + 2(qdp + f pdq) \\
    & = 2 f p^2 ds + 2 (f- 1) pdq \\
    & = \frac{2 p^2}{1+p^2} ((1- p^2) ds - 2 pdq ).
\end{align*}

This last observation concludes the proof of \cref{lem.U(-).-(U)} : 
\begin{align*}
LH(\La_x, V_0, \l_\eu)&  \simeq LH(\La_x, V_0, \l_\e) \simeq LH(\La_x, V_2, \l_\e) \\
  & \simeq LH(\La_x, V_2, \l_\eu) \simeq LH(\La_x, V, \l_\eu),
\end{align*}
which implies \cref{MainThm.UM-UK}, since the forms $\l_\eu$ and $\l_h$, used in \cref{sec.U(M-K)}, are both adapted to the contact manifold $V_0$.

\section{Corollaries and generalisations}

\subsection{The Floer homology of a fiber}\label{sec.stopped_Floer} 

In this last section we prove \cref{MainCor.stopped_Floer}. Consider $x \in S^3$ disjoint from the knot $K$, so the fiber at $x$ is $F \subset X = TS^3$, disjoint from the conormal $L_K$. Then the stopped Floer homology $CF(F; X, L_K)$ is a constructed by choosing a perturbed Lagrangian $\tilde F$, and considering the complex generated by 
\begin{itemize}
    \item the intersection points $F \cap \tilde F$ ;
    \item the Reeb chords going from the unit fiber $F^1$ to the Legendrian $\dd \tilde F$.
\end{itemize}
Usually $\tilde F$ is taken as $F$ perturbed by the Hamiltonian $H$, vanishing on the zero section and evaluating to $\e |p|^2$ away from it. In this case we obtain only one intersection point since $F \cap \tilde F = {(x,0)}$, and $\dd \tilde F = F^1 + \e R$.

We will however choose another perturbation as we want to avoid intersection points between the Lagrangians, so we can use the contact tools developed above. Hence we pick $\tilde F = F_{\tilde x}$ the fiber of a point close to $x$. Note that $\tilde F$ can still be obtained by Hamiltonian perturbation, since we have a path of exact Lagrangian interpolating between $F$ and $\tilde F$ (namely, we take the fibers over a path with endpoints $x$ and $\tilde x$).

With this perturbation we obtain no intersection points, and $\dd \tilde F = F^1_{\tilde x}$. By our previous construction, we the Legendrian homology of the fibers in the stopped manifold $US^3 \sms \La_K$ can be computed in $U(S^3 \sms K)$ :
\[LH(F^1_x, F^1_{\tilde x}  ; US^3 \sms \La_K) = LH(F^1_x, F^1_{\tilde x}  ; U(S^3 \sms K))\] 

Once again the Reeb chords from $F^1_x$ to $F^1_{\tilde x}$ correspond to elements of $\pi_K$. Note that this time we also get a chord corresponding to the unit, because we consider two unit fibers. Moreover the differential still vanishes by action reasons, and the product still corresponds to the group law by the same arguments, so we obtain the claimed isomorphism of rings :
\[ HF(F, DS^3 \sms \La_K) \simeq \ZZ[\pi_K]. \]

\subsection{Higher dimensional handles} \label{ssec.Higher_dim_handle}
The previous result still holds in a more general setting, it is only a matter of upgrading the notations. Given a $k$-dimensional submanifold $K \subset M$, where $M$ is of dimension $n$, choose coordinates on a neighbourhood
\[ \N(K) \simeq K_a \times D^{n-k}_z, \]
inducing coordinates on the unit bundle (after picking a metric on $K$) :
\[ \N(UM_{\rest K}) \simeq K_a \times D^{n-k}_z \times \{ (\a, \eta) \in T_x K \times \RR^{n-k},  \a^2 + \eta^2 = 1 \}.  \]
Note that the conormal of $K$ is now $\La_K = \{ z=0, \a = 0, \eta^2 = 1 \}$.

As previously, denote $V_0$ the unit bundle of the complement of (a standard neighbourhood of) $K$, and $V_1$ the complement of (a standard neighbourhood of) $\La_K$ :
\begin{align*}
    V_0 &   = \{\a^2 + \eta^2 = 1, z^2 \geq \d \} \\
    V_1 & = \{\a^2 + \eta^2 = 1,  z^2 + \a^2 \geq \d  \} ,
\end{align*}
for some small $\d >0$. Once again those contact manifolds have convex boundaries, and we can topologically obtain $V_1$ by gluing an handle 
\[H = \{\a^2 + \eta^2 = 1. z^2 \leq \d, \a^2 \leq \e\} \simeq UK \times D^{2(n-k)}\]
on the boundary of $V_0$, along the gluing locus
\[\{z^2 = \d, \a^2 = 1, \eta^2 = 0\} \simeq K \times S^{n-k-1} \times S^{k-1}.\]
 Note that the contact form $\l_0$, induced by the metric, is adapted to the dividing set of $\dd V_0$, given by 
\[ \Ga_0  = \{ z^2 = \d, z \cdot \eta = 0, \a^2 + \eta^2 = 1  \} \simeq K \times S^{n-k-1} \times S^{n-1}. \]

We encounter the same issues as in the previous section, namely the gluing locus of the handles intersects the dividing set, which we can once again solve by perturbing the contact form. Define
\[\l_\e  = (1+ \e \eta \cdot z) \l_0. \]
We will show that its dividing set is disjoint from the gluing locus of the handle, so the results of \cref{ssec.Circ_handle} applies. In this section we work without charts, which make the computations more straightforward. However the relation between $V_1$ and $V_0 \cup H$ is not so clear anymore, but working in charts as previously we can find a contact vector field interpolating between those two contact manifolds with convex boundary.

Let us first compute the Reeb vector field $R_\e$ of the contact form $\l_\e$, setting $f = 1+ \e z\cdot \eta$. It is related to $R_0$ by the formula
\[ R_\e = \frac{1}{f^2} (R_0 + \e f v), \] 
where $v \in \ker \l_0$ is determined by $\i_v d\l_0 = d(z\cdot \eta)$. A direct computation yields 
\[ v = -(z\cdot \eta) \dd_a + (z - (z\cdot \eta) \eta)\dd_z - \eta^2 \dd_\a - \a^2 \dd_\eta.   \]
Moreovoer the Reeb vector field of $\l_0$ is given by 
\[ R_0 = \a \dd_a + \eta \dd_z + * \dd_\a, \]
since we choosed the flat metric on the $\RR^{n-k}$ factor. We can now express the dividing set associated to $(V_0, \l_\e)$ :
\begin{align*}
	z.dz(R_\e)&  = \frac{1}{f^2} (z\cdot \eta + \e f z\cdot (z - (z\cdot \eta) \eta)) \\
	 & = \frac{1}{f^2} (X + \e (1+\e X) (\d - X^2)) \qquad \text{ where } X = z\cdot \eta \\
	 & =  \frac{1}{f^2} (-\e^2 X^3 - \e X^2 + (1+ \e^2 \d)X + \e \d).
\end{align*}

At the first order in $\e$, we get $ z.dz(R_\e) \sim -\e X^2 + X + \e \d + O(\e^2)$,
which vanishes when 
\[ X = \frac1{2\e} (-1 + \sqrt{1+4 \e^2 \d} ) + O(\e^2) =  \e \d + O(\e^2). \]

In other words $R_\e$ is tangent to $\dd V_0$ on  
\[\Ga_\e = \{z \cdot \eta  = \e \d + O(\e^2) \},  \]
which remains a contact manifold for $\e$ small enough. One can also check that the Liouville vector field induced on the boundary of $V_0$ is given (up to sign) by 
\[ Y =  (\eta + \frac{1}{z\cdot \eta} z) \dd_\eta + \a \dd_\a,\]
which is indeed transverse to $\Ga_\e$, so $\l_\e$ is adapted to the suture $\Ga_\e$.

Since the gluing locus of the handle is now disjoint from the suture, the arguments of \cref{sssec.Modify_contact_form} hold, which imply \cref{MainThm.UM-UK} :
\begin{align*}
    LH(\La, V_0, \l_0) & \simeq LH(\La, V_0, \l_\e) & \qquad \text{by invariance}\ \\
    & \simeq LH(\La, V_0 \cup H, \l_\e)   & \qquad \text{by \cref{ssec.Circ_handle}} \\
    & \simeq LH(\La, V_0 \cup H, \l_0)    & \qquad \text{by invariance}\  \\
    & \simeq LH(\La, V_1, \l_0),
\end{align*}
where the last step is obtained by interpolating between $V_0 \cup H$ and $V_1$, via a contact vector field, using the charts defined \cref{sssec.Stereo_charts}.

\subsection{Cooriented hypersurface} \label{ssec.Cooriented_hypersurf}

We now extend the previous result to the cooriented conormal hypersurface, proving \cref{MainThm.LCH_Hypersurf_Stop}. Consider an hypersurface $\Si \subset M$, where $M$ is $n$-dimensional, and choose local coordinates on a neighbourhood, inducing on the unit bundle
\[ \N(UM_{\rest \Si}) \simeq \Si_a \times [-1,1]_z \{(\a, \eta) \in T_a\Si \times \RR, \a^2 +  \eta^2 = 1 \}. \]

In those coordinates, the {\it cooriented} conormal of $\Si$ is 
\[ \La^+\Si = \{z= 0, \a = 0, \eta = 1  \},\]
and as before we denote $V_0$ the unit bundle of the complement of $\Si$, and $V_1$ the complement of the cooriented conormal $\La_\Si^+$.
Observe that topologically, one can recover $V_1$ from $V_0$ by gluing an handle $H= I \times D^* \Si$ along $\dd V_0$ ; more precisely, the gluing locus is $\{ z^2 = \d, \eta \leq \e_0, \a^2 + \eta^2 =1 \}$. We encounter the usual issue, as the dividing set $\{z^2 = \d, \eta=0, \a^2 = 1 \}$ is entirely contained in this part of the boundary.

Denote $M_0^\pm$ the two connected component of $M \sms \Si$, and $V_0^\pm$ there unit bundle. By perturbing the contact form as previously, we can displace the dividing set of $\dd V_0^-$ away from the gluing locus (however the dividing set of $\dd V_0^+$ is still included inside). In other words, we have
\begin{align*}
    V_1 & \simeq V_0 \cup H \\
        & \simeq V_0^- \cup (V_0^+ \cup H).
\end{align*}

By the result of \cref{ssec.Circ_handle}, for a Legendrian $\La \subset V_0^-$, we get
\[ LH(\La, V_1) \simeq LH(\La, V_0^-), \]
which proves \cref{MainThm.LCH_Hypersurf_Stop}.


\bibliographystyle{alpha}
\bibliography{references}

\newcommand{\etalchar}[1]{$^{#1}$}
\begin{thebibliography}{CGHH11}

\bibitem[Abo12]{Abo09_Fuk_loops}
Mohammed Abouzaid.
\newblock On the wrapped fukaya category and based loops.
\newblock {\em Journal of Symplectic Geometry}, 10(1), 2012.

\bibitem[AE21]{AsEk21_Chekanov-Eliashberg_dga_singular}
Johan Asplund and Tobias Ekholm.
\newblock Chekanov-eliashberg dg-algebras for singular legendrians, 2021.

\bibitem[AS06]{AbSc04_FH_cotangent}
Alberto Abbondandolo and Matthias Schwarz.
\newblock On the floer homology of cotangent bundles.
\newblock {\em Communications on Pure and Applied Mathematics}, 59(2):254--316,
  2006.

\bibitem[Asp21]{As21_Fiber_Floer_conormal_stops}
Johan Asplund.
\newblock Fiber floer cohomology and conormal stops.
\newblock {\em Journal of Symplectic Geometry}, 19(4):777–864, 2021.

\bibitem[Avd12]{Avd12_Connect_sum_cobord}
Russell Avdek.
\newblock Liouville hypersurfaces and connect sum cobordisms.
\newblock 2012.

\bibitem[BEH{\etalchar{+}}03]{BEHWZ03_Compactness_SFT}
Frederic Bourgeois, Yakov Eliashberg, Helmut Hofer, Kris Wysocki, and Eduard
  Zehnder.
\newblock Compactness results in symplectic field theory.
\newblock {\em Geometry \& Topology}, 7(2), 2003.

\bibitem[BKO19]{BKO19_Formal_boundary_hyperb_knot_compl}
Youngjin Bae, Seonhwa Kim, and Yong-Geun Oh.
\newblock Formality of floer complex of the ideal boundary of hyperbolic knot
  complement.
\newblock 2019.

\bibitem[Bou09]{Bou09_Survey_contact}
Fr{\'e}d{\'e}ric Bourgeois.
\newblock A survey of contact homology.
\newblock {\em New perspectives and challenges in symplectic field theory},
  49:45--71, 2009.

\bibitem[BP12]{BenedPetro12_Hyperbolic}
R.~Benedetti and C.~Petronio.
\newblock {\em Lectures on Hyperbolic Geometry}.
\newblock Universitext. Springer Berlin Heidelberg, 2012.

\bibitem[CELN17]{CELN16_Knot_&_cord_alg}
Kai Cieliebak, Tobias Ekholm, Janko Latschev, and Lenhard Ng.
\newblock Knot contact homology, string topology, and the cord algebra.
\newblock {\em Journal de l’École polytechnique — Mathématiques}, 4,
  2017.

\bibitem[CGH10]{CGH10_ECH_&_OBD}
Vincent Colin, Paolo Ghiggini, and Ko~Honda.
\newblock Embedded contact homology and open book decompositions.
\newblock 2010.

\bibitem[CGHH11]{CGHH10_Sutures}
Vincent Colin, Paolo Ghiggini, Ko~Honda, and Michael Hutchings.
\newblock Sutures and contact homology i.
\newblock {\em Geometry \& Topology}, 15(3), 2011.

\bibitem[CH05]{CoHo05_Constr_control_Reeb}
Vincent Colin and Ko~Honda.
\newblock Constructions de champs de reeb sous contrôle et applications.
\newblock {\em Geometry \& Topology}, 2005.

\bibitem[Che97]{Che02_Dga_leg_links}
Yuri Chekanov.
\newblock Differential algebras of legendrian links.
\newblock 1997.

\bibitem[Dat20]{Dat20_PhD_Homologies_leg_sut_applications}
C{\^o}me Dattin.
\newblock {\em Homologies legendriennes sutur{\'e}es et applications {\`a} la
  construction conormale}.
\newblock PhD thesis, Universit{\'e} de Nantes, 2020.

\bibitem[Dat22]{Dat22_Sutured_Legendrian_homology_2-braids}
C{\^o}me Dattin.
\newblock Wrapped sutured legendrian homology and the conormal of 2-braids,
  2022.

\bibitem[EENS13]{EENS11_Knot_hom}
Tobias Ekholm, John~B Etnyre, Lenhard Ng, and Michael~G Sullivan.
\newblock Knot contact homology.
\newblock {\em Geometry \& Topology}, 17(2), 2013.

\bibitem[EES02]{EES02_LCH_R2n+1}
Tobias Ekholm, John Etnyre, and Michael~G. Sullivan.
\newblock Legendrian submanifolds in $\mathbf{R}^{2n+1}$ and contact homology,
  2002.

\bibitem[EES07]{EES05_LCH_PxR}
Tobias Ekholm, John Etnyre, and Michael Sullivan.
\newblock Legendrian contact homology in ${P} \times \mathbf{R}$.
\newblock {\em Transactions of the American Mathematical Society}, 359(07),
  2007.

\bibitem[Ekh07]{Ek05_Morse_flow_trees}
Tobias Ekholm.
\newblock Morse flow trees and legendrian contact homology in 1–jet spaces.
\newblock {\em Geometry \& Topology}, 11(2), 2007.

\bibitem[Ekh08]{Ekh06_Rational_SFT_Z2_cobord}
Tobias Ekholm.
\newblock Rational symplectic field theory over $\zz_2$ for exact lagrangian
  cobordisms.
\newblock {\em Journal of the European Mathematical Society}, 2008.

\bibitem[ENS17]{ENS16_Complete_knot_invrt}
Tobias Ekholm, Lenhard Ng, and Vivek Shende.
\newblock A complete knot invariant from contact homology.
\newblock {\em Inventiones mathematicae}, 211(3), 2017.

\bibitem[Fab07]{Fab07_Obstruction_bundles_Action_filtration}
Oliver Fabert.
\newblock Obstruction bundles over moduli spaces with boundary and the action
  filtration in symplectic field theory, 2007.

\bibitem[Gab83]{Gab83_Folia}
David Gabai.
\newblock Foliations and the topology of 3-manifolds.
\newblock {\em J. Differential Geom.}, 18(3):445--503, 1983.

\bibitem[Gir91]{Gir91_Convexity}
Emmanuel Giroux.
\newblock {\em Convexit{\'e} en topologie de contact}.
\newblock PhD thesis, Lyon 1, 1991.

\bibitem[Gir03]{Gir02_Dim3_vers+}
Emmanuel Giroux.
\newblock Géométrie de contact: de la dimension trois vers les dimensions
  supérieures.
\newblock 2003.

\bibitem[GL89]{GoLu89_Knots_complements}
C.~McA. Gordon and J.~Luecke.
\newblock Knots are determined by their complements.
\newblock {\em Journal of the American Mathematical Society}, 2(2):371--415,
  1989.

\bibitem[GPS19]{GPS17_Fuk_sectors}
Sheel Ganatra, John Pardon, and Vivek Shende.
\newblock Covariantly functorial wrapped floer theory on liouville sectors.
\newblock {\em Publications mathématiques de l’IHÉS}, 2019.

\bibitem[LR68]{LaRh68_Ring_group_left_order}
R.H. LaGrange and A.H. Rhemtulla.
\newblock A remark on the group rings of order preserving permutation groups.
\newblock {\em Canadian Mathematical Bulletin}, 11(5), 1968.

\bibitem[Ng08]{Ng04_Framed_knot_hom}
Lenhard Ng.
\newblock Framed knot contact homology.
\newblock {\em Duke Mathematical Journal}, 141(2), 2008.

\bibitem[Par19]{Par15_CH_&_virtual}
John Pardon.
\newblock Contact homology and virtual fundamental cycles.
\newblock {\em Journal of the American Mathematical Society}, 32(3), 2019.

\bibitem[PR20]{PaRu20_Aug_&_Immers_lagr_fill}
Yu~Pan and Dan Rutherford.
\newblock Augmentations and immersed lagrangian fillings, 2020.

\bibitem[Rol14]{Rol14_Order_group}
Dale Rolfsen.
\newblock Low-dimensional topology and ordering groups.
\newblock {\em Mathematica Slovaca}, 64(3), 2014.

\bibitem[She19]{She16_Conormal_torus_knot_inrt}
Vivek Shende.
\newblock The conormal torus is a complete knot invariant.
\newblock In {\em Forum of Mathematics, Pi}, volume~7. Cambridge University
  Press, 2019.

\bibitem[Syl19]{Syl16_Stops}
Zachary Sylvan.
\newblock On partially wrapped fukaya categories.
\newblock {\em Journal of Topology}, 12(2), 2019.

\bibitem[Thu79]{Thu79_3-mfds}
William~P Thurston.
\newblock {\em The geometry and topology of three-manifolds}.
\newblock Princeton University Princeton, NJ, 1979.

\bibitem[Wal68]{Wal68_Irred_3-mfds}
Friedhelm Waldhausen.
\newblock On irreducible 3-manifolds which are sufficiently large.
\newblock {\em Annals of Mathematics}, pages 56--88, 1968.

\end{thebibliography}

    \else
    \fi

\end{document}